\documentclass{amsart}
\usepackage{amsmath, amsfonts, amssymb, url}
 \usepackage[all]{xy}
\usepackage{enumerate} 

\usepackage{hyperref}

\newcommand{\mB}{\mathcal{B}}

\newcommand{\mG}{\mathcal{G}}

\newcommand{\mJ}{\mathcal{J}}

\newcommand{\mL}{\mathcal{L}}
\newcommand{\mM}{\mathcal{M}}

\newcommand{\mS}{\mathcal{S}}
\newcommand{\mT}{\mathcal{T}}

\newcommand{\mV}{\mathcal{V}}

\newcommand{\bfA}{\mathbf{A}}

\newcommand{\bfC}{\mathbf{C}}

\newcommand{\bfF}{\mathbf{F}}

\newcommand{\bfQ}{\mathbf{Q}}

\newcommand{\bfT}{\mathbf{T}}

\newcommand{\bfZ}{\mathbf{Z}}

\newcommand{\Oo}{\mathcal{O}}

\newcommand{\ov}{\overline}

\newcommand{\be}{\begin{equation}}
\newcommand{\ee}{\end{equation}}
\newcommand{\bes}{\begin{equation*}}
\newcommand{\ees}{\end{equation*}}

\newcommand{\bs}{\begin{split}}
\newcommand{\es}{\end{split}}
\newcommand{\bss}{\begin{split*}}
\newcommand{\ess}{\end{split*}}

\newcommand{\bmat}{\left[ \begin{matrix}}
\newcommand{\emat}{\end{matrix} \right]}
\newcommand{\bsmat}{\left[ \begin{smallmatrix}}
\newcommand{\esmat}{\end{smallmatrix} \right]}

\newcommand{\bml}{\begin{multline}}
\newcommand{\eml}{\end{multline}}
\newcommand{\bmls}{\begin{multline*}}
\newcommand{\emls}{\end{multline*}}

\DeclareMathOperator{\End}{End}

\DeclareMathOperator{\Frob}{Frob}
\DeclareMathOperator{\Gal}{Gal}
\DeclareMathOperator{\GL}{GL}
\DeclareMathOperator{\GSp}{GSp}

\DeclareMathOperator{\Hom}{Hom}

\DeclareMathOperator{\Sel}{Sel}

\DeclareMathOperator{\val}{val}

\newcommand{\tr}{\textup{tr}\hspace{2pt}}

\usepackage[all]{xy}
\usepackage{amsthm}
\usepackage{amssymb, amsfonts}
\SelectTips{cm}{10}\UseTips
\theoremstyle{plain}
\newtheorem{thm}{Theorem}
\newtheorem{prop}[thm]{Proposition}
\newtheorem{cor}[thm]{Corollary}
\newtheorem{lemma}[thm]{Lemma}

\theoremstyle{definition}
\newtheorem{definition}[thm]{Definition}

\newtheorem{rem}[thm]{Remark}

\numberwithin{thm}{section}
\numberwithin{equation}{section}

\author{Tobias Berger and Krzysztof Klosin}
\title[Irreducibility of limits of Galois representations]{Irreducibility of limits of Galois representations of Saito-Kurokawa type}

\begin{document}
 
\thanks{The first author's research was supported by the EPSRC Grant EP/R006563/1. The second author was supported by a Collaboration for Mathematicians Grant \#578231  from the Simons Foundation
 and by a PSC-CUNY award jointly funded by the Professional Staff Congress and the City
University of New York.}

\begin{abstract} We prove (under certain assumptions) the irreducibility of the limit $\sigma_2$  of a sequence of irreducible  essentially self-dual Galois representations $\sigma_k: G_{\bfQ} \to \GL_4(\ov{\bfQ}_p)$ (as $k$ approaches 2 in a $p$-adic sense) which mod $p$ reduce (after semi-simplifying) to $1 \oplus \rho \oplus \chi$ with $\rho$ irreducible, two-dimensional of determinant $\chi$, where $\chi$ is the mod $p$ cyclotomic character.  More precisely, we assume that $\sigma_k$ are crystalline (with a particular choice of weights) and Siegel-ordinary at $p$. Such representations arise in the study of $p$-adic families of Siegel modular forms and properties of  their limits as $k\to 2$ appear to be important in the context of the Paramodular Conjecture.  The result is deduced from the finiteness of two Selmer groups whose order is controlled by $p$-adic $L$-values of an elliptic modular form (giving rise to $\rho$) which we assume are non-zero.   \end{abstract}
\maketitle

\section{Introduction}
In \cite{BergerKlosinAJMaccepted} the authors studied the modularity of abelian surfaces with rational torsion. Let $A$ be an abelian surface over $\bfQ$, let $p$ be a prime and suppose that $A$ has  a rational point of order $p$, and a polarization of degree prime to $p$. Then the (semi-simplified)  action of $G_{\bfQ}:=\Gal(\ov{\bfQ}/\bfQ)$ on $A(\ov{\bfQ})[p]$ is of the form $1 \oplus \rho \oplus \chi$, for $\chi$ the mod $p$ cyclotomic character. Assuming that $\rho$ is irreducible, Serre's conjecture (Theorem of Khare-Wintenberger) implies that the mod $p$ representation looks like the reduction of that of a Saito-Kurakawa lift of an elliptic modular form $f$ of weight 2. If $\End(A)=\bfZ$ then the $p$-adic Tate module of $A$ gives rise to an irreducible $p$-adic Galois representation.  The Paramodular Conjecture (formulated by Brumer and Kramer \cite{BrumerKramer14}) predicts that this representation should be isomorphic to the Galois representation attached to a weight 2 Siegel modular form of paramodular level which is not in the space of Saito-Kurokawa lifts. Establishing the modularity of $A$ by a Siegel modular form therefore requires proving congruences between the Saito-Kurokawa lift $SK(f)$ and ``non-lifted type (G)"
Siegel modular forms. The latter are cuspforms staying cuspidal under the transfer to $\GL_4$, and are expected to be exactly the forms whose associated $p$-adic representation is irreducible.

Such congruences for Saito-Kurokawa lifts have been proven by Brown, Agarwal and Li \cite{Brown07}, \cite{AgarwalBrown14}, [Brown-Li] for holomorphic Siegel modular forms of congruence level $\Gamma_0^{2}(N)$ and paramodular level $\Gamma_{\rm para}(N)$ for weights $k$ larger than 6 (see [Brown-Li] Corollary 6.15). With this new result \cite{BergerKlosinAJMaccepted}  Theorem 10.2 can be generalized to allow ramification at a squarefree level $N$, and establishes a so-called $R=T$ result and the modularity of Fontaine-Laffaille representations that residually are of Saito-Kurokawa type (with an elliptic $f$ of weight $2k-2$ for $k \geq 6$).
A different type of congruences have also been constructed by Sorensen, see section \ref{paramod}.

The methods used to prove these congruences unfortunately do not extend to weight $k=2$, the case of interest for the modularity of abelian surfaces. We  propose to use $p$-adic families to prove the relevant congruences in weight 2 (albeit a priori only to a $p$-adic modular form - see below).
 For example, Skinner and Urban \cite{SkinnerUrban06} proved that for an ordinary elliptic form $f$ the  $\Gamma_{\rm para}(N)$-level holomorphic  Saito-Kurokawa lift $SK(f)$ can be $p$-adically interpolated by a semi-ordinary (also called Siegel-ordinary) family. 
It is plausible that their arguments could be adapted for $\Gamma_0^2(N)$-level holomorphic Saito-Kurokawa lifts. Such $p$-adic families have also  been studied by Kawamura \cite{Kawamura10preprint} and Makiyama \cite{Makiyama18}.

As part of a work in progress we construct (under some assumptions) another Siegel-ordinary $p$-adic family (of tame level either  $\Gamma_0^{2}(N)$ or  $\Gamma_{\rm para}(N)$) interpolating  the type of congruences constructed by Brown or Sorensen. At classical weights $k \gg 0$ its points  would correspond to irreducible $p$-adic Galois representations that are Siegel-ordinary  (see Definition \ref{S-ord def}) and whose semi-simplified residual representation is the mod $p$ representation associated to $SK(f)$.

One could then use this family to approach weight 2 via weights $k \gg 0$, but $k \to 2$ $p$-adically. As points of weight 2  for such a family are critical (in the sense that the $U_p=U_{p,1} U_{p,2}$-slope is at least one and therefore does not satisfy the small slope condition in Theorem 7.1.1 of \cite{AndreattaIovitaPilloni15}; see section \ref{Siegel1} for definitions of $U_{p,1}$ and $U_{p,2}$) 
 it is not clear whether this limit would correspond to a  classical Siegel modular form. 

In fact, modularity by $p$-adic Siegel modular forms was proved for certain abelian surfaces whose $p$-adic Galois representation is residually irreducible by Tilouine \cite{Tilouine06}. 
In a sense this paper provides a necessary ingredient to proving such  $p$-adic modularity for the residually reducible case as explained below. Let us also mention that some strong potential modularity results in the residually irreducible situation have recently been proven in \cite{BCGP}.

One potential problem is that while the $p$-adic Galois representations attached to the members of the family for $k \gg 0$ are irreducible this is not a priori clear of the limit. This property is on the one hand necessary for modularity purposes (as  $T_pA\otimes \bfQ_p$ is irreducible). On the other hand it allows one then to feed these ingredients into a machinery similar to the one developed in \cite{BergerKlosinAJMaccepted} (modified appropriately for representations that are Siegel-ordinary instead of Fontaine-Laffaille)  and under suitable conditions show that $T_pA$ and the limit Galois representation are in fact isomorphic, thus proving $p$-adic modularity of $A$. 

In this paper we introduce a new way of proving that under certain assumptions the limit of irreducible Galois representations is itself irreducible. This method is based on finiteness of Selmer groups and while we only apply it here in our specific situation (i.e., when the representations are residually of Saito-Kurokawa type, as  desired for proving the modularity of abelian surfaces with rational $p$-torsion) it is not difficult to see how it can be modified to work in other contexts,  cf. our upcoming paper about a residually reducible $R=T$ result for $\GL_2$ in weight 1.

In other words, while our overarching goal is to provide ingredients to prove modularity of abelian surfaces as explained above,  the theorems proven in  this paper could in principle be treated completely independently as a result on limits of Galois representations. In particular,  Siegel modular forms will be notably absent from our statements and their presence will manifest itself only through certain conditions imposed on the Galois representations. 
We thus consider a family (which is part of a ``refined'' rigid analytic family in the sense of Balla\"iche-Chenevier - see section \ref{Main Theorem}) of irreducible 4-dimensional $p$-adic Galois representations $\sigma_k$ indexed by a set of integers $k >2$,  $k\equiv 2$ (mod ($p-1$)) which approach 2 in the $p$-adic sense. Suppose that $\tr \sigma_k$ converge $p$-adically to some pseudo-representation $T$ when $k \to 2$. We require that for each $k$ the representation $\sigma_k$ reduces to some mod $p$ representation whose semi-simplification is isomorphic to $1 \oplus \chi \oplus \rho$ for an irreducible 2-dimensional representation $\rho$ and that it is crystalline and Siegel-ordinary. We are interested in conditions guaranteeing the irreducibility of $T$.

The basic idea is not difficult to explain. First we use the irreducibility of $\sigma_k$ to construct Galois stable lattices in their representation spaces so that infinitely many of the $\sigma_k$s reduce mod $p$ to a non-semi-simple residual representation (whose semi-simplification is $1 \oplus \chi \oplus \rho$) with the same Jordan Holder factor as a subrepresentation and the same Jordan-Holder factor as a quotient. It is not possible to ensure that  \emph{all} $\sigma_k$ reduce to the same combination as $\ov{\sigma}_k$ has three Jordan-Holder factors. Indeed, in general Ribet's Lemma only tells us that there are enough (non-split) extensions between different Jordan-Holder factors to guarantee connectivity of a certain graph - see  section \ref{SK type} - and absent any other assumptions (like for example lying in the Fontaine-Laffaille range which was used in Corollary 4.3 of \cite{BergerKlosinAJMaccepted}) there is no way to tell which extension will arise. However, as  there are only finitely many such extensions possible, we get an infinite subsequence $\mT$ of $\sigma_k$ with identical (non-split)  reduction.

Now, if $T$ was reducible, 
there are several ways in which it can split into the sum of irreducible pseudo-representations. Let us discuss here the case of three Jordan-Holder factors which can be regarded as the main result of this paper - see Theorem \ref{SK case}.  In that case  as $k\in \mT$ approaches $2$ ($p$-adically) the representations $\sigma_k$ become reducible modulo  $p^{n_k}$ with $n_k$ tending to $\infty$. As the reduction of $\sigma_k$ is non-split, we conclude that $\sigma_k$ give rise to elements in a certain Selmer group of arbitrary high order. Using symmetries built into the Galois representation one shows that this Selmer group can only be one of two possibilities. Then the Main Conjecture of Iwasawa Theory gives us that the orders of these Selmer groups are controlled by  specializations  to weight 2 (at two different points) of a certain $p$-adic $L$-function. 
Hence to guarantee that these Selmer groups are finite (i.e., that $T$ cannot be reducible) we impose a non-vanishing condition on these $L$-values. As we a priori do not know for which of the possible extensions we get the infinite subsequence $\mT$ we need to control  both of the $L$-values as above.  See section \ref{SK type} for details.

A priori if $T$ is reducible it could also split into 2 or 4 components and we deal with them in sections \ref{Main Theorem} and \ref{Yoshida type}. We are able to rule out all of them, albeit for the reduction type dealt with in section \ref{Yoshida type}, the so called Yoshida type, our theorems require quite strong assumptions.

We would like to thank Adel Betina, Pol van Hoften, Chris Skinner, and Ariel Weiss for helpful discussions related to the topics of this article and Andrew Sutherland for the example in section \ref{paramod}.

\section{Setup} \label{Setup} 
 Let $p$ be an odd prime. Let $E$ be a finite extension of $\bfQ_p$ with integer ring $\Oo$, uniformizer $\varpi$ and residue field $\bfF$.  We fix an embedding $\ov{\bfQ}_p \hookrightarrow \bfC$. Write $\epsilon$ for the $p$-adic cyclotomic character and $\chi$ for its mod $\varpi$ reduction.  Let $N$ be a square-free positive integer with $p\nmid N$. Let $\Sigma$  be the set of primes of $\bfQ$ consisting of $p$ and the primes dividing $N$. We denote by $G_{\Sigma}$ the Galois group of the maximal Galois extension of $\bfQ$ unramified outside of the set $\Sigma$. 

Consider a Galois representation $\rho: G_{\Sigma} \to \GL_2(\bfF)$ of which we assume that it is odd and  absolutely irreducible of determinant $\chi$. 
Furthermore we assume that $\rho$ is ordinary and $p$-distinguished, i.e., $$\rho|_{D_p} \cong \bmat \eta^{-1}\chi & * \\ &\eta\emat,$$ where $\eta$ is a non-trivial unramified character and that  $\rho|_{I_p} $ is non-split.
 We further assume that $\rho$ is ramified at all primes dividing $N$ and that $\rho|_{I_{\ell}}$ has a fixed line for all $\ell \mid N$ (or equivalently that $N$ is the prime-to-$p$-part of the conductor of $\rho$).

Let $\tau: G \to \GL_n(\Oo)$ be an $n$-dimensional representation of a group $G$ or $\tau: \Oo[G] \to \Oo$ be an $n$-dimensional pseudo-representation of $G$. For a definition of a pseudo-representation, its dimension  and  basic properties we refer the reader to section 1.2.1 of \cite{BellaicheChenevierbook}. However, let us only mention here that an $n$-dimensional pseudo-representation $\tau$ is called \emph{reducible} if $\tau = \tau_1 + \tau_2$ for some pseudo-representations $\tau_1, \tau_2$ (each necessarily of dimension smaller than $n$). A pseudo-representation that is not reducible is called \emph{irreducible}. In particular, if $\tau: G\to \GL_n(\Oo)$ is a representation, then  $T:=\tr \tau$ is an $n$-dimensional pseudo-representation and $T$ is reducible if and only if $\tau$ is. Furthermore if $\tau$ is an $n$-dimensional pseudo-representation and $\tau= \sum_{i=1}^{r} \tau_i$ with each $\tau_i$ an irreducible pseudo-representation, then this decomposition as a sum of irreducible pseudo-representations is unique (up to reordering of the summands). 

Now let $G=G_{\Sigma}$.
 By composing a representation or pseudo-representation $\tau$ with the reduction map $\Oo \to \bfF$ we obtain the \emph{reduction} of $\tau$ which we will denote by $\ov{\tau}$. If $\tau$ is an $n$-dimensional representation valued in $\GL_n(E)$, one can always find a $G_{\Sigma}$-stable $\Oo$-lattice $\Lambda$ such that when we choose a basis of $E^n$ to be a basis of $\Lambda$ we obtain a representation $\tau_{\Lambda}$ valued in $\GL_n(\Oo)$. The isomorphism class of $\tau_{\Lambda}$ and also of its reduction $\ov{\tau}_{\Lambda}$ depends in general on the choice of $\Lambda$. However, the semi-simplification  $\ov{\tau}^{\rm ss}_{\Lambda}$ (and hence also the pseudo-representation $\tr \ov{\tau}_{\Lambda}$) is independent of $\Lambda$ and so it makes sense to drop $\Lambda$ from the notation.

\begin{lemma} \label{lattice1} Let $\tau: G_{\Sigma} \to \GL_n(E)$ be a continuous representation and let $V$ be the representation space of $\tau$. Suppose that there exists a subspace $L \subset V$ of dimension $r \leq n$ with the following two properties: $L$ is stable under $G_{\Sigma}$ and  $G_{\Sigma}$ acts on $L$ via an irreducible representation $\psi: G_{\Sigma} \to \GL_r(E)$ with values in $\GL_r(\Oo)$. Let $\Lambda$ be a $G_{\Sigma}$-stable $\Oo$-lattice in $V$ ($\Lambda\otimes_{\Oo}E=V$). Then $\Lambda$ has a rank $r$ free $\Oo$-submodule which is stable under $G_{\Sigma}$ and on which $G_{\Sigma}$ acts via the representation $\psi$. \end{lemma}

\begin{proof} Fix a basis $\mB=\{e_1, e_2, \dots, e_n\}$ for the $\Oo$-module $\Lambda$. Then $\mB$ is also a basis for the vector space for $V$.  Let $v =\alpha_1 e_1 + \dots + \alpha_n e_n\in L$ be a non-zero vector. Then there exists $s \in \bfZ_{\geq 0}$ such that $\varpi^s \alpha_i\in \Oo$ for all $i=1,2, \dots, n$. Thus $0 \neq v_0:=\varpi^s v \in \Lambda \cap L$. Set $\Lambda_0:=\Oo$-span of $\{g \cdot v_0 \mid g \in G_{\Sigma}\}\subset L$. Clearly $\Lambda_0$ is an $\Oo$-module, but note that it is also stable under $G_{\Sigma}$. Indeed, let $v':=\beta_1 g_1 \cdot v_0 + \dots+ \beta_k g_k \cdot v_0 = \beta_1 \psi(g_1)v_0 + \dots + \beta_k \psi(g_k) v_0 \in \Lambda_0$ with $\beta_1, \dots, \beta_k \in \Oo$, $g_1, \ldots g_k \in G_{\Sigma}$. Let $g \in G_{\Sigma}$. Then $g \cdot v' = \psi(g) \beta_1 \psi(g_1)v_0\dots + \psi(g) \beta_k \psi(g_k) v_0 = \beta_1 \psi(gg_1)v_0\dots +\beta_k \psi(gg_k) v_0 \in \Lambda_0$, the first equality being true since $v'\in L$, i.e., $G_{\Sigma}$ acts on $\Lambda_0$ via $\psi$. Furthermore, since $\psi(g)$ has entries in $\Oo$ for every $g \in G_{\Sigma}$ and $v_0 \in \Lambda$, we conclude that $\Lambda_0$ is an $\Oo$-submodule of $\Lambda$. Hence, as $\Oo$ is a PID, we get that $\Lambda_0$ is a free finitely generated $\Oo$-module of rank $r' \leq n$.  
  Finally, $\Lambda_0 \otimes_{\Oo} E \subset L$ is a non-zero subspace of $L$ which is stable under the action of $G_{\Sigma}$, i.e., is a non-zero subrepresentation of $\psi$. Since $\psi$ is irreducible we must have $\Lambda_0 \otimes_{\Oo}E = L$ and so $r'=r$. 
\end{proof}

\begin{lemma} \label{subtoquo} Let  $\tau: G_{\Sigma} \to \GL_n(E)$ an irreducible representation. Suppose that with respect to some $G_{\Sigma}$-stable $\Oo$-lattice $\Lambda$ of the representation space $V$ of $\tau$ one has $\ov{\tau}_{\Lambda} \cong \bmat \tau_1 & * \\& \tau_2\emat $ for $\tau_i: G_{\Sigma}\to \GL_{r_i}(\bfF)$, $r_1+r_2=n$. Then there exists a $G_{\Sigma}$-stable $\Oo$-lattice $\Lambda'$ of the representation space $V$  such that with respect to $\Lambda'$ we have $\ov{\tau}_{\Lambda'} \cong \bmat\tau_1 \\ * &\tau_2\emat$. 
\end{lemma}

\begin{proof} For $g \in G_{\Sigma}$ write $\tau_{\Lambda}(g)=\bmat a_g & b_g \\ c_g & d_g \emat$. Then $c_g$ is an $r_2 \times r_1$ matrix whose entries we denote by $c_{ij}(g)$. Let $S=\{ g \in G_{\Sigma} \mid c_g \neq 0\}$. Irreducibility of $\tau$ guarantees that $S$ is non-empty. For $g \in S$ set $m_g:= \min\{\val_{\varpi}(c_{ij}(g))\mid \textup{$i,j$ such that $c_{ij}(g) \neq 0$}\}$. Furthermore set $m = \min_{g \in S}m_g$ and note that $m\geq 1$ as $\ov{\tau}_{\Lambda}$ is upper-triangular. Then $$\bmat 1 \\ & \varpi^{-m}\emat \bmat a_g & b_g \\ c_g & d_g \emat \bmat 1 \\ & \varpi^{m}\emat = \bmat a_g & \varpi^m b_g \\ \varpi^{-m} c_g & d_g \emat.$$
\end{proof}

In this article we will be especially interested in 2-dimensional and 4-dimensional Galois representations that are \emph{ordinary} in a sense that we now define. 
\begin{definition}\label{S-ord def}\begin{enumerate}  \item A Galois representation $\tau: G_{\Sigma} \to \GL_2(E)$ will be called \emph{ordinary} if $\tau|_{D_p} \cong \bmat \psi^{-1} \epsilon^{k-1}&* \\ & \psi\emat$ for some positive integer $k$ and some unramified character $\psi$.    
\item A Galois representation $\tau: G_{\Sigma} \to \GL_4(E)$ will be called \emph{Siegel-ordinary} if $$\tau|_{D_p} \cong \bmat  \psi^{-1}\epsilon^{2k-3} & * & * & * \\ & * & * & * \\ &*&*&*\\ &&&\psi\emat,$$ for some positive integer $k$ and some unramified Galois character $\psi$.
\item A Galois representation $\tau: G_{\Sigma} \to \GL_4(E)$ will be called \emph{Borel-ordinary} if $$\tau|_{D_p} \cong \bmat  \psi^{-1}\epsilon^{2k-3} & *& * & * \\ &  \phi^{-1} \epsilon^{k-1} & * & * \\ & &\phi \epsilon^{k-2}&*\\ &&&\psi\emat,$$ for some positive integer $k$ and some unramified Galois characters $\psi$ and $\phi$.
\end{enumerate}
\end{definition}

For later it will be useful to introduce the following notation. If $\alpha \in E^{\times}$, then the unramified character from $D_p$ to $E^{\times}$ that takes the arithmetic Frobenius to $\alpha$ will be denoted by $\phi_{\alpha}$. 

\section{Irreducibility} \label{Main Theorem}

\subsection{Main assumptions} \label{Main assumptions} 
Assume we have a $p$-adic family of Galois representations in the sense of \cite{BellaicheChenevierbook}, i.e. we have a rigid analytic space $X$ over $\bfQ_p$ and a $4$-dimensional pseudo-representation $\bfT: G_{\Sigma} \to \Oo(X)$.  We denote by $\sigma_x:G_{\Sigma} \to \GL_4(E(x))$ (for some finite extension $E(x)$ of $\bfQ_p$) the semi-simple representation of $G_{\Sigma}$ whose trace is the evaluation $\bfT_x$ of $\bfT$ at $x \in X$ (for existence see \cite{Taylor91}, Theorem 1). We are interested in the case when the family satisfies nice $p$-adic Hodge properties for all points in a Zariski dense set $Z \subset X$ and want to deduce properties at a point $x_0 \in X \backslash Z$, in particular control the ramification at $p$ of the corresponding Galois representation. The reader should think of $X$ as (an affinoid subdomain of) an eigenvariety parametrizing Siegel modular forms. We therefore also assume the existence of a weight morphism $w: X \to \mathcal{W}$, where $\mathcal{W}$ is the rigid analytic space over $\bfQ_p$ such that $\mathcal{W}(\bfC_p)=\Hom_{\rm cts}((\bfZ_p^{\times})^2, \bfC_p^{\times})$.

More precisely, assume that we have data $(X, \bfT, \{\kappa_n\}, \{F_n\}, Z)$, a refined family in the sense of \cite{BellaicheChenevierbook} Definition 4.2.3, where $n=1, \ldots 4$ and $\kappa_n$ and $F_n$ are analytic functions in $\Oo(X)$.  For $z \in Z$ we have $0=\kappa_1(z) < \kappa_2(z)<\kappa_3(z) < \kappa_4(z)$ are the Hodge-Tate weights of $\sigma_z$. The case of  interest to us is  where for a point $z$ of weight $w(z)=(w_1, w_2)$ with $w_1\geq w_2$ we have $\kappa_2(z) = w_2-2$, $\kappa_3(z)=w_1-1$ and $\kappa_4(z) = w_1+w_2+3$.  We assume $\sigma_z$ is crystalline   and the eigenvalues of $\varphi$ on $D_{\rm cris}(\sigma_z)$ are given by $(p^{\kappa_1(z)}F_1(z), \ldots, p^{\kappa_4(z)}F_4(z))$. Furthermore, suppose there exists an involution $\tau:\Oo(X)[G_{\Sigma}] \to \Oo(X)[G_{\Sigma}]$ given by $\tau(g) = \Phi(g)g^{-1}$ for some character $\Phi: G_{\Sigma} \to \Oo(X)^{\times}$ with  $\Phi|_{D_p} = \epsilon^{\kappa_4(z)}$ such that  $\bfT \circ \tau=\bfT$.

We also assume that for $z \in Z$ the representation  $\sigma_z|_{D_p}$ is Siegel-ordinary, i.e. that  $$\sigma_z|_{D_p} \cong \bmat  \psi^{-1}\epsilon^{\kappa_4(z)} & * & * & * \\ & * & * & * \\ &*&*&*\\ &&&\psi\emat.$$ This is equivalent to demanding that $|F_1(z)|=1$ and then $\psi=\phi_{F_1(z)}$. The existence of $\tau$ then implies that $F_4(z)=F_1(z)^{-1}$. In addition we assume that $\sigma_z$ is $p$-distinguished, i.e., $\ov{\psi} \neq 1$.

 Fix $x_0\in X\setminus Z$ of weight $w(x_0)=(2,2)$ and from now we reserve the notation $E$ for the field $E(x_0)$ and denote by $\Oo$ the ring of integers in $E$ with uniformizer $\varpi$ and residue field $\bfF$.  Put $T=\bfT_{x_0}$ and $\sigma_2:= \sigma_{x_0}$. We assume that $T \equiv 1 + \tr(\rho) + \chi \mod{\varpi}$ for $\rho$ as in section \ref{Setup} and that $F_2(x_0) \neq 0$.

Let $\mS$ be a sequence of integers $k \equiv 2$ (mod $p^{m_k-1}(p-1)$)  with $m_k \to \infty$ as $k \to \infty$.  
We single out a sequence of points $z_k \in Z$  converging to $x_0$ with $w(z_k)=(k,k)$ for $k \in \mS$.
Denote the corresponding family of Galois representations $\sigma_k:=\sigma_{z_k}: G_{\Sigma} \to \GL_4(E_k)$, where we set $E_k:= E(z_k)$. Extending $E_k$ if necessary we may assume that $\Oo\subset \Oo_k$, where $\Oo_k$ is the ring of integers of $E_k$ with uniformizer $\varpi_k$. Then we define $n_k \in \bfZ_{\geq 0}$ to be  the largest integer $n$ such that $\tr \sigma_k \equiv T$ mod $\varpi^{n}$. Note the convergence $z_k \to x_0$ implies $n_k \to \infty$ as $k\to \infty $ but approaches 2 $p$-adically.

 We assume that for each $k\in \mS$  the representations $\sigma_k$ have the following properties (of which (2), (3) and (5) follow from the assumption made on $\bfT$ and so does (4) for $k\gg 0$, but we record them here again for the ease of reference): 
\begin{enumerate}
\item $\sigma_k$ is irreducible,
\item $\det \sigma_k = \epsilon^{4k-6}$,
\item $\sigma_k^{\vee} \cong \sigma_k(3-2k)$, 
\item $\ov{\sigma}_k^{\rm ss}\cong 1 \oplus \rho \oplus \chi$,  
\item  $\sigma_k|_{D_p}$ is crystalline with weights $2k-3, k-1, k-2, 0$ and $\sigma_k$ is Siegel-ordinary at $p$, i.e., $$\sigma_k|_{D_p} \cong\bmat  \phi_{\beta_k}^{-1}\epsilon^{2k-3} & * & * & * \\ & * & * & * \\ &*&*&*\\ &&&\phi_{\beta_k}\emat,$$ for $\beta_k \in \Oo_k^{\times}$ and we assume that $\beta_k \not\equiv 1$ mod $\varpi_k$,
i.e., $\ov{\sigma}_k$ is $p$-distinguished;
\item If $\ell \in \Sigma - \{p\}$ then $\ell \not \equiv 1$ (mod $p$) or $\sigma_k|_{I_{\ell}}$ is unipotent.
\end{enumerate}
We refer the reader to Remark \ref{Siegel1} for a relation between these properties of $\sigma_k$ and Siegel modular forms.

\begin{lemma} \label{HT}We have  \begin{enumerate}[(i)] \item $T|_{D_p}=\phi_{\beta}^{-1} \epsilon + \phi_{\beta} +\tr \gamma$ for $\beta = F_1(x_0)$ and a continuous representation $\gamma:D_p \to {\rm GL}_2(\Oo)$.
\item The pseudo-representation $T$ (or rather $\sigma_2$) has Hodge-Tate-Sen weights  0,0,1,1.  
\item  Furthermore, if  $\Psi$ is any character that occurs in the decomposition of $T|_{D_p}$ into pseudo-representations then we must have $\Psi|_{I_p} = \epsilon$ or $\Psi|_{I_p}=1$
\end{enumerate}
\end{lemma}
\begin{proof}
 For (i) we use the Siegel-ordinarity of the $\sigma_z$ for $z \in Z$ and continuity. 

For (ii) we apply \cite{BellaicheChenevierbook} Lemma 7.5.12 and deduce that  the Hodge-Tate-Sen weights in weight 2 are 0,0,1,1. 

For (iii) first note that the statement is clear if  $\Psi =\phi_{\beta}$ or $\Psi=\phi_{\beta}^{-1}\epsilon$. So we now consider the case when$\gamma|_{D_p}^{\rm ss}=\Psi \oplus \Psi'$ for some character $\Psi'$. Part (ii) tells us that $\Psi$ is Hodge-Tate of weight 0 or 1, so equal to a finite order character (not necessarily unramified) or the product of such a character and $\epsilon$. We want to use the crystallinity of $\sigma_z$ for $z\in Z$ to deduce that $\Psi$ is crystalline. Results of Kisin and Bella\"iche-Chenevier allow to continue crystalline periods for the smallest Hodge-Tate weight. Note that either $\phi_{\beta}$ or $\phi_{\beta}^{-1}\epsilon$ has the same Hodge-Tate weight as $\Psi$.  To be able to attribute the crystalline period to $\Psi$ (rather than $\phi_{\beta}$ or $\phi_{\beta}^{-1}\epsilon$) we use the Siegel-ordinary and $p$-distinguishedness assumptions we made on $\sigma_z$ for $z \in Z$:

As in \cite{BergerBetina19} proof of Theorem 5.3 (which uses geometric Frobenius convention, so considers representations dual to the ones we have here) we can quotient out the sheaf $\mM$ corresponding to $\Oo(X)[D_p]/\ker \bfT$ (cf. \cite{BellaicheChenevierbook} Lemma 4.3.7) by a subsheaf $\mL$ corresponding to the line stabilised by $I_p$ on which $\Frob_p$ acts by $F_4 p^{\kappa_4}$. The quotient sheaf $\widetilde{\mM/\mL}$ is generically of rank 3  and its semi-simplification specializes at $x_0$ to $\Psi \oplus \Psi' \oplus \phi_{\beta}$. As in the proof of \cite{BergerBetina19} Proposition 8.2 Siegel-ordinarity further tells us that $\widetilde{\mM/\mL}$ has a torsion-free subsheaf $\mM'$ of generic rank 2 such that the specialisations $\sigma'_z$ at $z \in Z$ are 2-dimensional crystalline representations with Hodge-Tate weights $\kappa_2(z), \kappa_3(z)$ and with crystalline period for the appropriate Hodge-Tate weight, i.e. $D_{\rm cris}(\sigma'_z)^{\varphi=F_i(z) p^{\kappa_i(z)}} \neq 0$ for $i=2$ or $3$. (Note that for $k \in \mS$ we have $\kappa_2(z_k)=k-2$ and $\kappa_3(z_k)=k-1$.) The semi-simplification of the sheaf $\mM'$ specializes at $x_0$ to  $\Psi \oplus \Psi'$.

Applying (the torsion-free analogue of) \cite{BellaicheChenevierbook} Theorem 3.3.3(i) to $\mM'$ then gives that $D_{\rm cris}(\mM')^{\varphi=F_i(x_0)p^{\kappa_i(x_0)}} \neq 0$. Since by assumption $F_2(x_0) \neq 0$ (and so also $F_3(x_0)\neq 0$) this means that one of the characters $\Psi$ or $\Psi'$ is crystalline, so equal to a power of the cyclotomic character times a finite order unramified character. As discussed before this power must be 0 or 1. As $T|_{D_p} = T|_{D_p}\circ \tau$ with $ \tau(g) = \epsilon(g) g^{-1}$ we get $\Psi\Psi'=\epsilon$.  So we are done. 
\end{proof}

\subsection{Possible splitting types  of $T$}
Now suppose that $T$ is reducible. Then $T$ is in one of the following cases:
\begin{itemize}
\item[(i)] $T=T_1+T_2 +T_3 +T_4$, where each $T_i$ is a character;
\item[(ii)] $T=T_1+ T_2 + T_3$, where $T_1$ and $T_3$ are characters and $T_2$ is an irreducible pseudo-representation of dimension 2 (we refer to this type of splitting as the \emph{Saito-Kurokawa type});
\item[(iii)] $T=T_1+T_2$, where $T_1$, $T_2$ are both irreducible pseudo-representations of dimension 2 (we refer to this type of splitting as the \emph{Yoshida type});
\item[(iv)] $T=T_1 + T_2$, where $T_1$ is an irreducible pseudo-representation of dimension 3 and $T_2$ is a character.
\end{itemize}

\begin{prop} \label{i and iv}  Cases (i) and (iv) cannot occur.
\end{prop}

\begin{proof} Case (i) cannot occur because $\ov{\sigma}_k^{\rm ss}\cong 1 \oplus \rho \oplus \chi$ for every $k\in \mS$, so also $\ov{T}= 1 + \tr \rho + \chi$ and $\rho$ is irreducible (so also $\tr \rho$ is irreducible as a pseudo-representation). 

Let us now show that $T$ is not as in case (iv). 
Suppose $T$ is as in case (iv). Then $T = \xi + \tr \rho_0$, where $\xi: G_{\Sigma} \to \Oo^{\times}$ is a character and $\rho_0$ is a 3-dimensional irreducible representation.
As $T=T\circ \tau$, we must have $\xi|_{I_p} =\epsilon \xi|_{I_p}^{-1}$. This contradicts Lemma \ref{HT}(iii). 
\end{proof}

For an ordinary newform $g=\sum_{n=1}^{\infty} a_n(g)q^n$ of weight 2 let $L(g,s)$ denote the standard $L$-function of $g$ and let $L_p(g,2)$ be the $p$-adic $L$-value denoted by $ L_p^{\rm an}(g, \omega^{-1}, T=p)$   in section 2 of \cite{BergerKlosinAJMaccepted}.
The proof of the following theorem will be given in the next section.

\begin{thm} \label{SK case} Assume that $\rho|_{G_K}$ is absolutely irreducible for $K=\bfQ(\sqrt{(-1)^{(p-1)/2}p})$. Suppose that $L(g,1)L_p(g,2)\neq 0$ for all $p$-ordinary  newforms  $g$ of weight 2  and level dividing $N^2p$  such that 
$a_{\ell}(g) \equiv \tr\rho(\Frob_{\ell}) $ mod $\varpi$ for all primes $\ell \nmid Np$.  
 Then $T$ is not of Saito-Kurokawa type. 
\end{thm}

Note that there are only finitely many (possibly none) forms $g$ as in Theorem \ref{SK case}.

In section \ref{Yoshida type} we discuss some conditions that guarantee that $T$ is not of Yoshida type either. All these results combined would guarantee that $T$ is in fact irreducible, however, the assumptions allowing us to rule out the Yoshida type are quite strong (cf. Remark \ref{limitation}).

\section{Siegel modular forms and Paramodular Conjecture}
\subsection{Siegel modular forms} \label{Siegel1}
We recall some facts about Siegel modular forms and their associated Galois representations to motivate our discussion. 
By Arthur's classification (see \cite{Arthur04} and \cite{GeeTaibi19}) cuspidal automorphic representations for ${\rm GSp}_4(\bfA_{\bfQ})$ fall into different types. Cuspidal automorphic representations $\pi$ whose transfer to $\GL_4$ stays cuspidal are called of ``general type" or type (G). Such type (G) representations are expected to have irreducible $p$-adic Galois representations (see \cite{Weiss19} for a summary of what's known and results in the low weight case).  Other types in the classification are known to be associated to reducible $p$-adic Galois representations, see \cite{BCGP}  Lemma 2.9.1. Particular examples of such types are the Saito-Kurokawa lifts and Yoshida lifts of elliptic modular forms, whose associated Galois representations have trace of Saito-Kurokawa or Yoshida type respectively. Schmidt \cite{SchmidtCAP} proved that holomorphic Siegel modular forms of paramodular level are either of type (G) or Saito-Kurokawa lifts, while other CAP types or Yoshida lifts do not occur. 

We denote by $U_{p,1}$ (resp. $U_{p,2}$) the Hecke operators associated to ${\rm diag}(1,1,p,p)$ (resp. ${\rm diag}(1,p,p^2,p)$). 
For $\pi$ of sufficiently high weight (i.e. corresponding to classical Siegel eigenforms of weights $k_1 \geq k_2 \geq 3$) we have the following result about properties of the associated Galois representations (for a more detailed statement see \cite{BCGP} Theorem 2.7.1):

\begin{thm}[Laumon, Weissauer, Sorensen, Mok, Faltings-Chai, Urban]
Suppose $\pi$ is a cuspidal automorphic representation for $\GSp_4(\bfA_{\bfQ})$ of weight $k_1 \geq k_2 \geq 3$. Then there is a continuous semi-simple representation $\rho_{\pi}: G_{\bfQ} \to \GSp_4(\ov{\bfQ}_p)$ with $$\rho_{\pi}^{\vee} \cong \rho_{\pi}(3-k_1-k_2)$$ satisfying the following properties:

\begin{enumerate}
\item For each prime  $\ell \neq p$  we have local-global compatibility up to semi-simplification with the local Langlands correspondence proved by Gan-Takeda. In particular,  if $\pi$ is unramified at $\ell$ then so is $\rho_{\pi}$ and  if $\pi$ is of Iwahori level at $\ell$ then $\rho_{\pi}|_{I_\ell}$ is unipotent. 
\item If $\rho_{\pi}$ is irreducible then for each prime $\ell \neq p$  one has local-global compatibility up to Frobenius semi-simplification. 
\item $\rho_{\pi}|_{D_p}$ is de Rham with Hodge-Tate weights $k_1+k_2-3, k_1-1, k_2-2, 0$.

\item Assume that $\pi$ is  Siegel-ordinary at $p$ (i.e $\lambda_{p,1}$ is a $p$-adic unit, $\lambda_{p,2}$ has finite $p$-valuation, where $\lambda_{p,i}$ is the $U_{p,i}$-eigenvalue of $\pi$ for $i=1,2$), then $\rho_{\pi}|_{D_p}$ is Siegel-ordinary in the sense of Definition 2.4 with the unramified character having $\lambda_{p,1}$ as value at $\Frob_p$.
\item If $\pi$ is unramified at $p$ then the $p$-adic representation $\rho_{\pi}$ is crystalline at $p$. If $\pi$ is also Siegel-ordinary  then the characteristic polynomial of Frobenius acting on $D_{\rm cris}(\rho_{\pi}|_{D_p})$ equals the Hecke polynomial. In particular, the eigenvalues are $$\lambda_{p,1}, \lambda_{p,1}^{-1}\lambda_{p,2} p^{k_2-2}, \lambda_{p,1}\lambda_{p,2}^{-1}p^{k_1-1}, \lambda_{p,1}^{-1}p^{k_1+k_2-3}.$$

\end{enumerate}
\end{thm}

Suppose now that $\rho$  as in section \ref{Setup} equals $\ov{\rho}_f$ for $f \in S_2(Np)$. 
 If $f$ is ordinary it lies in a Hida family of eigenforms $f_k$. Brown et al. \cite{Brown07}, \cite{AgarwalBrown14}, \cite{BrownLi19preprint} then prove that there exists a family of holomorphic Siegel modular eigenforms $F_k$ for $k \in \mathcal{S}$ with $\mS$ as in section 3 of Iwahori level $N$ (level $\Gamma_0^{(2)}(N)$ or $\Gamma_{\rm para}(N)$ )  that are congruent to the Saito-Kurokawa lifts $SK(f_k)$ modulo $p$  and $\sigma_{F_k}$ is irreducible (see e.g. \cite{AgarwalBrown14} Corollary 7.5). The theorem above then shows that the associated Galois representations $\sigma_{F_k}$ satisfy the conditions (1)-(6) in section \ref{Main assumptions} except possibly for Siegel-ordinarity  (note that we assume that  $\rho$ is $p$-distinguished). 
To establish that the $\tr \sigma_{F_k}$ interpolate $p$-adically is work in progress. 

The pseudo-representation of the (Siegel-ordinary, tame level $N$) eigenvariety then gives rise to $\bfT:G_{\Sigma} \to \Oo(X)$ for an affinoid $X$ containing the limit point $x_0$ of weight $(2,2)$. One obtains a Zariski dense subset $Z\subset X$ of classical points that are old at $p$ such that $(X, \bfT, \{\kappa_n\}, \{F_n\}, Z)$ is a refined family in the sense of Bella\"iche-Chenevier. By the above theorem the function $F_1=F_4^{-1}$ interpolates the $U_{p,1}$-eigenvalue  $\lambda_{p,1}$, $F_2=F_3^{-1}$ interpolates $\lambda_{p,1}^{-1}\lambda_{p,2}$, so our assumption $F_2(x_0)\neq 0$ corresponds to the $U_{p,2}$-slope of the limit form being finite.

\subsection{Discussion of applicability to the Paramodular Conjecture} \label{paramod}
 For an elliptic modular form $f$ of weight $2k-2$ a holomorphic Saito-Kurokawa lift exists under the following conditions on $f$ and $k$: for $\Gamma_0^{2}(N)$-level  $k$  has to be even, for $\Gamma_{\rm para}(N)$-level the sign of the functional equation of $f$ has to be $-1$ (see \cite{Schmidt07}). 

Suppose $\rho=\ov{\rho}_f$ for an ordinary newform $f$ of level $N$.  For Theorem \ref{SK case} we need to assume that $L(f,1) \neq 0$. Continuing our discussion from the introduction about Saito-Kurokawa congruences, we note that in the case that $L(f,1) \neq 0$  we would therefore need to consider congruences with holomorphic $\Gamma_0^{2}(N)$-level Saito-Kurokawa lifts. However, a different method to the one used by Brown et al.  (pointed out to us by Pol van Hoften) could be used to prove the required congruences for paramodular level: Using the arguments from the proof of \cite{Sorensen09} Theorem D one should be able to prove congruences for the generic (as opposed to the holomorphic) Saito-Kurokawa lift, for which the conditions on $k$ and the root number are reversed. 

 Once the congruence between the generic Saito-Kurokawa lift and a type (G) form has been proved, one could then switch to the holomorphic element of the same packet.  
If such a congruence could be proved in weight 2 this would also explain the example of the abelian surface of conductor 997 mentioned in \cite{BergerKlosinAJMaccepted}  (which involves an elliptic modular form $f$ with root number $\epsilon=1$ and $L(f,1)=0$).

 To demonstrate that examples with $L(f,1) \neq 0$ occur when studying the modularity of abelian surfaces  we thank Andrew Sutherland for providing us with the following abelian surface: Let  $A$ be the Jacobian of the genus 2 curve
$$C: y^2 + (x + 1)y = -2x^6 + x^5 - x^4 + 9x^3 - 2x^2 + 2x - 9$$
(see \cite[\href{http://www.lmfdb.org/Genus2Curve/Q/1870/a/226270/1}{Genus 2 Curve 1870.a}]{lmfdb} and \cite{BSSVY}).
Then $A$ has conductor $1870 = 2*5*11*17$ and comparing values on $\Frob_{\ell}$ for $\ell<10^6$ indicates that $$A(\ov{\bfQ})[3] \cong 1 \oplus \ov{\rho}_f \oplus \chi$$  for $f$
the unique weight 2 newform of level $\Gamma_0(17)$ corresponding
to the isogeny class of rank 0 elliptic curves over $\bfQ$ with conductor 17.

\section{Ruling out Saito-Kurokawa type} \label{SK type}
We keep the notation and assumptions of section \ref{Setup}, \ref{Main assumptions} and Theorem \ref{SK case}. In this section we will prove Theorem \ref{SK case}.
Recall that by assumption (4) we have  $\ov{\sigma}_k^{\rm ss} = 1 \oplus \rho \oplus \chi$ for every $k \in \mS$.

\begin{lemma} \label{inde} Let $k\in \mS$. Then there exists a $G_{\Sigma}$-stable lattice $\Lambda$ in the representation space of $\sigma_k$ such that  \begin{itemize}
\item[(i)]
$$\ov{\sigma}_{k, \Lambda} = \bmat 1 & \ov{a}_k & \ov{b}_k \\ \ov{d}_k & \rho & \ov{c}_k \\ \ov{e}_k & \ov{f}_k & \chi\emat$$ and  for at least one $\ov{x}_k \in \{\ov{a}_k, \ov{c}_k , \ov{d}_k, \ov{f}_k\}$, the representation $\ov{\sigma}_{k, \Lambda}$ has a subquotient isomorphic to a non-semisimple representation of the form $\sigma_{\ov{x}_k}:=\bmat\ov{\rho}_1 & \ov{x}_k \\ & \ov{\rho}_2\emat$, where $\ov{\rho}_1, \ov{\rho}_2 $ are distinct elements of $\{1,\rho, \chi\}$, one of which equals $\rho$, and are determined by  $\ov{x}_k$.
\item[(ii)] $\ov{\sigma}_{k, \Lambda}$ is indecomposable.
\end{itemize}
\end{lemma}
\begin{proof}  Consider the graph $\mG$ whose vertices are elements of the set $\mV=\{1, \rho, \chi\}$ and where we draw a \emph{directed} edge from $\rho' \in \mV$ to $\rho''\in \mV$ if there exists a $G_{\Sigma}$-stable lattice $\Lambda'$ such that  $\ov{\sigma}_{k, \Lambda'}$ has a subquotient isomorphic to a non-semi-simple representation of the form $\bmat\rho' & x \\ & \rho''\emat$. Then by a theorem of Bella\"iche for any two $\rho', \rho''\in \mV$, there exists a directed path from $\rho'$ to $\rho''$ (see Corollaire 1 in \cite{Bellaiche03}). In particular there must be at least one edge originating at $\rho$  and at least one edge ending at $\rho$. One of these edges gives rise to the desired $\sigma_{\ov{x}_k}$. This proves (i). If for such a $\Lambda'$ the representation $\ov{\sigma}_{k, \Lambda'}$ is already indecomposable then we are done. Otherwise $\ov{\sigma}_{k, \Lambda'}$ is the direct sum of $\sigma_{\ov{x}_k}$ and the remaining element $v$ of $\mV$. Hence we can  apply Theorem 4.1 in \cite{BergerKlosinAJMaccepted} with $\sigma_{\ov{x}_k}$ as a quotient and $v$ as a (non-semi-simple) subrepresentation. This gives us a new lattice $\Lambda$ so we have $\ov{\sigma}_{k, \Lambda}$ is a non-split extension of $\sigma_{\ov{x}_k}$ by $v$, hence indecomposable. 
\end{proof}
\begin{cor} \label{cor5.2} Under the same assumptions as in Lemma \ref{inde} the subquotient $\sigma_{\ov{x}_k}$ can be taken to be a subrepresentation or a quotient, i.e., there exists at least one lattice for each option.
\end{cor} 
\begin{proof} This follows from Lemma \ref{subtoquo}. \end{proof}

For the rest of the section  assume that $T= T_1 + T_2+T_3$ with $T_1, T_2, T_3$ where $\Psi_1:=T_1$ and $\Psi_2:=T_3$ are characters and $T_2$ is two-dimensional and irreducible. We assume that $\ov{\Psi}_1 = 1$, $\ov{\Psi}_2=\chi$ and $\ov{T}_2=\tr \rho$. Our goal is to show that these assumptions lead to a contradiction, and thus prove Theorem  \ref{SK case}. Since $T_2$ is irreducible we get by \cite{Taylor91} Theorem 1 that $T_2=\tr \tilde{\rho}$ for some irreducible 2-dimensional representation $\tilde{\rho}: G_{\Sigma} \to \GL_2(E)$ reducing to $\rho$.

\begin{lemma}\label{ordrhotilde} The representation $\tilde{\rho}$ is ordinary. \end{lemma}
\begin{proof} By Lemma \ref{HT} we have  $\sigma_2 |_{D_p}^{\rm ss} = \phi_{\beta}^{-1} \epsilon\oplus \phi_{\beta} \oplus \gamma$, where $\gamma$ is two-dimensional. Since $\beta \not\equiv 1$ mod $\varpi$ by our assumption (5), we cannot have $\Psi_1|_{D_p}, \Psi_2|_{D_p} \in \{ \phi_{\beta}^{-1} \epsilon, \phi_{\beta}\}$. Hence it must be the case that $\tilde{\rho}|_{D_p}^{\rm ss} \cong \phi_{\beta}^{-1}\epsilon \oplus \phi_{\beta}$. Suppose $\tilde{\rho}|_{D_p} \cong \bmat \phi_{\beta} & * \\ 0 & \phi_{\beta}^{-1}\epsilon\emat$. Note that $\ov{\tilde{\rho}}\cong \rho$ is irreducible, so in particular well-defined and we have by assumption that $\rho|_{D_p}$ does not have an unramified subrepresentation of dimension 1. Thus neither can $\tilde{\rho}|_{D_p}$. Hence we get that $\tilde{\rho}|_{D_p} \cong \bmat \phi_{\beta}^{-1}\epsilon & * \\ 0 & \phi_{\beta}\emat$ as desired.
\end{proof}

Recall that for every $k\in \mS$ we write $n_k$ for the largest integer such that $\tr \sigma_k \equiv T$ (mod $\varpi^{n_k}$).   Note that under the assumptions from section \ref{Main assumptions} one clearly has  $n_k\to \infty$ as $k$ approaches 2 $p$-adically.
\begin{lemma} \label{LemmaUrban} Let $k \in \mS$, $\mJ=\{\Psi_1, \tilde{\rho}, \Psi_2\}$ and let $\Lambda$ be a lattice from Lemma \ref{inde} so that $$\ov{\sigma}_k:= \ov{\sigma}_{k, \Lambda} = \bmat \ov{\rho}_3 & * & * \\ & \ov{\rho}_1 & \ov{x}_k \\ && \ov{\rho}_2\emat$$ is indecomposable with non-semi-simple 3-dimensional quotient $\bmat  \ov{\rho}_1 & \ov{x}_k \\ & \ov{\rho}_2\emat$ (cf. Corollary \ref{cor5.2}).  Then $$\sigma_k \cong_{\Oo_k} \bmat \rho_3 & y_k & z_k \\ & \rho_1 & x_k \\ && \rho_2 \emat \pmod{\varpi^{n_k}}.$$  Here $\rho_i$ are distinct elements of $\mJ$ and $\rho_i = \ov{\rho}_i$ mod $\varpi$ and $x_k = \ov{x}_k$ mod $\varpi$. In particular the class $[x_k] \in H^1(\bfQ, \Hom(\rho_2, \rho_1)\otimes \Oo_k/\varpi^{n_k})$ has the property that $\varpi^{n_k-1}[x_k] \neq 0$.
\end{lemma}
\begin{proof} This follows from Remarks (a) and (d) in \cite{Urban99} (cf. also Theorem 1.1 in \cite{Brown08}). The last statement follows directly from the fact that the quotient $\bmat \ov{\rho}_1 & \ov{x}_k \\ & \ov{\rho}_2\emat$ is not semi-simple. \end{proof}

\begin{lemma} \label{l5.6} There exists an ordinary newform $g$ of weight 2 and level dividing $N^2p$ such that $\tilde{\rho} = \rho_g$.\end{lemma}
\begin{proof} By Lemma \ref{ordrhotilde} we have that $\tilde{\rho}|_{D_p} \cong \bmat \phi_{\beta}^{-1}\epsilon & * \\ 0 & \phi_{\beta}\emat$, i.e., $\tilde{\rho}$ is an ordinary deformation of $\rho$. In particular, its Hodge-Tate weights are 1 and 0. Furthermore, the assumption that $\rho|_{G_K}$ be absolutely irreducible (with $K$ as in Theorem \ref{SK case}) guarantees that  $\tilde{\rho}$ is modular by some ordinary newform $g$ of weight 2 by a generalization of a theorem of Wiles due to Diamond - see Theorem 5.3 in \cite{Diamond96}.  The $p$-part of the level of $g$ is $p$ or 1 (see e.g., Lemma 3.26 in \cite{DDT}). For primes $\ell \mid N$ such that $\ell \not \equiv 1 \mod{p}$  the discussion in section 10.2 of \cite{BergerKlosin13} explains how a result of Livne implies that the level of $g$ is at most $\ell^2$.  If $\ell \mid N$ and $\ell \equiv 1 $ (mod $p$), then the level is, in fact, at most $\ell$ due to our unipotency assumption.  For this note that if $V$ denotes the space corresponding to $\tilde \rho$ this means that $V^{I_\ell}$ is 1-dimensional. As we are also assuming that the residual reduction $V_{\rho}^{I_{\ell}}$ is 1-dimensional and that $\rho$ has Artin conductor $\ell$, the Artin conductors of $\rho$ and $\tilde \rho$ agree and equal $\ell$ (as their valuations are given by $\dim V -\dim V^{I_{\ell}}+{\rm sw}(\tilde \rho)$  and $\dim V_{\rho} -\dim V_{\rho}^{I_{\ell}}+{\rm sw}(\rho)$, respectively,  and ${\rm sw}(\rho)={\rm sw}(\tilde \rho)$ by Serre). 
\end{proof}
\begin{rem} \begin{enumerate} \item The reader may note that if no $g$ as in the statement of Theorem \ref{SK case} exists then Lemma \ref{l5.6} already gives a contradiction to the assumption that $T$ is of Saito-Kurokawa type. \item Similar analyses of reducibility ideals for families approximating
holomorphic paramodular Saito-Kurokawa lifts were carried out in \cite{SkinnerUrban06}
and \cite{BergerBetina19}  in characteristic zero (necessarily under different
assumptions, in particular for $L(g, 1) = 0$). In the following we present
arguments working  in characteristic $p$. However, it is possible that a characteristic zero approach would also yield our result.
 \end{enumerate} \end{rem}

Write $V_g$ for the representation space of $\rho_g$ and let $V_g^+\subset V_g$ be the one-dimensional subspace on which $I_p$ acts via $\epsilon$. Let $T_g \subset V_g$ be any $G_{\Sigma}$-stable lattice in $V_g$.   
The following Lemma follows from the fact that any two $G_{\Sigma}$-stable lattices are homothetic.
\begin{lemma} \label{lattice isom} Let $\tau: G_{\Sigma} \to \GL_2(E)$ be residually irreducible. Let $\Lambda, \Lambda'$ be two $G_{\Sigma}$-stable lattices in the representation space of $\tau$. Then $\tau_{\Lambda} \cong \tau_{\Lambda'}$ (over $\Oo$). In other words, $\Lambda$ and $\Lambda'$ are isomorphic as $\Oo[G_{\Sigma}]$-modules, i.e., there exists $M \in \GL_2(\Oo)$ such that $\tau_{\Lambda'} = M \tau_{\Lambda} M^{-1}$. 
\end{lemma}

In particular, the action of $G_{\Sigma}$ on $T_g/\varpi T_g$ (which we denote by $\ov{\rho}_{g, T_g}$)  is isomorphic to $\ov{\rho}_g\cong \rho$ as the latter representation is irreducible. Furthermore, by Lemma \ref{lattice isom} we get that the isomorphism class of the restriction of the action of $G_{\Sigma}$ to $I_p$ on $T_g$ is independent of the choice of $T_g$ inside the representation space of $\rho_g$. More precisely, we have the following result.

\begin{lemma} \label{indep1} One has $\rho_{g, T_g}|_{I_p} \cong_{\Oo} \bmat \epsilon & * \\ & 1\emat.$ 
\end{lemma}
\begin{proof} By Lemma \ref{lattice isom} it is enough to show that there exists a $G_{\Sigma}$-stable lattice $\Lambda_0$ such that $\rho_{g, \Lambda_0}|_{I_p} = \bmat \epsilon &x \\ & 1 \emat.$ For this see proof of Proposition 6 of \cite{Ghate05}. 
\end{proof}

Write $W_g$ for $V_g/T_g \cong \rho_{g,T_g} \otimes E/\Oo$. By Lemma \ref{indep1} we know that there exist rank one free $\Oo$-submodules $T_g^+$ and $T_g^-$ of $T_g$ such that $T_g = T_g^+ \oplus T_g^-$ as $\Oo$-modules and that if $e_1\in T_g^+$ and $e_2\in T_g^-$ form a basis of $T_g$ then in the basis $\{e_1, e_2\}$ one has $\rho_{g, T_g}|_{I_p} = \bmat\epsilon & x \\ & 1 \emat $ with $x \not \equiv 0$ mod $\varpi$ (as $\ov{\rho}_g|_{I_p}=\rho|_{I_p}$ is non-split).   One clearly has $T_g^+\otimes_{\Oo}E = V_g^+$. Set $W_g^+:= V_g^+/T_g^+ \cong T_g^+\otimes_{\Oo}E/\Oo$. 

Following \cite{SkinnerUrban06} 3.1.3 we define Greenberg-style Selmer groups $${\rm Sel}_i: = \ker\left(H^1(G_{\Sigma} , W\otimes\Psi_i^{-1}) \xrightarrow{{\rm res}_{I_p}} H^1(I_p, (W_g/W_g^+) \otimes \Psi_i^{-1})\right), i=1, 2.$$

\begin{lemma} \label{Psis} One has $\Psi_1=1$ and $\Psi_2=\epsilon$. 
\end{lemma}
\begin{proof}  By assumption (6) we know that $\Psi_1$ and $\Psi_2$ are unramified away from $p$. Since $\ov{\Psi}_1=1$ and $\ov{\Psi}_2=\chi$ we know by Lemma \ref{HT}(iii) that $\Psi_1$ is unramified everywhere, hence trivial. As $\Psi_1 \Psi_2=\epsilon$ we get $\Psi_2=\epsilon$.
\end{proof}

\begin{prop} \label{Sel finite}  The groups ${\rm Sel}_i$ are finite for $i=1,2$. \end{prop}
\begin{proof} Recall $$L(g, s) =\prod_{\ell \nmid N} (1-a_{\ell}(g)\ell^{-s} + \ell^{-2s+1})^{-1} \prod_{\ell\mid N}(1-a_{\ell}(g)\ell^{-s})^{-1}\quad \textup{for ${\rm Re}(s)\gg 0.$}$$  Let $L^{N}(g,s)$ be defined in the same way but omitting the Euler factors at primes $\ell \mid N$. By Theorem 4.6.17 in \cite{Miyake89} we get that the $\ell$-eigenvalue $a_{\ell}(g)$ of $g$ equals $0$ or $\pm 1$,  hence $1-a_{\ell}(g)\ell^{-i} \neq 0$ for $i=1,2$. This implies that $L(g,i) \neq 0$ if and only if $L^{N}(g, i) \neq 0$ for $i \in \{1,2\}$.   By \cite{SkinnerUrban14}  Theorem 3.36 we have $\#{\rm Sel}_1 \leq \#\Oo/L^N(g,1)$. In the notation of \cite{SkinnerUrban14} we are in the case $m=0$ and $a_p(g)-1 \in \Oo^{\times}$ due to our $p$-distinguishedness assumption on $\rho$.   For $i=2$  we use the argument from the proof of \cite{BergerKlosinAJMaccepted} Proposition 2.10. In particular, by the Main Conjecture of Iwasawa theory and the control theorem (\cite{BergerKlosinAJMaccepted} Theorem 2.11) we deduce $$\#{\rm Sel}_2 \leq\# \Oo/L^N_p(g,2).$$  \end{proof}

As the representations $\sigma_{k, \Lambda}$ are valued in $\Oo_k$, rather than $\Oo$ we need to introduce some auxiliary Selmer groups. For $k \in \mS$ and $r \in \bfZ_+$
we set $${\rm Sel}_{i,k,r}: = \ker\left(H^1(G_{\Sigma} , T_{g,k,r}\otimes\Psi_i^{-1}) \xrightarrow{{\rm res}_{p}} H^1(I_p, (T_{g,k, r}/T_{g,k, r}^+) \otimes \Psi_i^{-1})\right), i=1, 2,$$ where $T^?_{g, k,r}=T^?_g\otimes \Oo_k/\varpi^r\Oo_k$ for $?\in \{+, \emptyset\}$.
Note that for $k=2$ we have an inclusion  ${\rm Sel}_{i,2,r} \hookrightarrow {\rm Sel}_i[\varpi^{r}]$.
\begin{prop}\label{InSel}  If $\ov{x}_k\in \{\ov{d}_k, \ov{f}_k\}$, then $[x_k] \in {\rm Sel}_{1, k, n_k}$. If $\ov{x}_k\in \{\ov{a}_k, \ov{c}_k\}$, then $[x_k] \in {\rm Sel}_{2, k, n_k}$. In either case $\varpi^{n_k-1}[x_k] \neq 0$. 
\end{prop} 
\begin{proof} Write $$\sigma_k = \bmat \Psi_1 & a & b \\ d & \tilde{\rho} & c \\ e & f & \Psi_2 \emat \pmod{\varpi^{n_k}}$$ as before with $a_k=\bmat a_k^1 & a_k^2\emat$, $d_k = \bmat d_k^1 & d_k^2\emat^t$, $c_k=\bmat c_k^1 & c_k^2 \emat^t$ and $f_k=\bmat f_k^1 & f_k^2\emat$. By Siegel-ordinarity we have $$\sigma_k|_{D_p} \cong_{E_k} \bmat\phi_{\beta}^{-1} \epsilon &*&*&*\\ &*&*&*\\ &*&*&*\\ &&&\phi_{\beta}\emat.$$ 
 Furthermore, by Lemma \ref{ordrhotilde} we have $\tilde{\rho}|_{D_p} = \bmat \phi_{\beta}^{-1} \epsilon & h \\ & \phi_{\beta} \emat.$  Thus in particular $$(\sigma_k|_{D_p} \pmod{\varpi^{n_k}})^{\rm ss} = \Psi_1 \oplus \Psi_2 \oplus \phi_{\beta}^{-1} \epsilon\oplus \phi_{\beta} \pmod{\varpi^{n_k}}.$$ Conjugating $\sigma_k$ by a permutation matrix we see that $$\sigma_k|_{D_p} \cong_{\Oo_k} \bmat\phi_{\beta}^{-1} \epsilon  & d_k^1 & c_k^1 & h\\ a_k^1 & \Psi_1 & b_k & a_k^2\\ f_k^1 & e_k & \Psi_2 & f_k^2\\0&d_k^2& c_k^2 & \phi_{\beta}\emat\pmod{\varpi^{n_k}}.$$ 
To complete the proof of Proposition \ref{InSel} we need several lemmas.
\begin{lemma} \label{vanishing} One has 
\begin{itemize}
\item  If $\ov{x}_k =\ov{a}_k$, then $a_k^1$ gives rise to an  extension of $D_p$-modules $\bmat \Psi_1 & a_k^1 \\ & \phi^{-1}_\beta\epsilon\emat$ mod $\varpi^{n_k}$, which splits, i.e., $[a_k^1]=0$. 
\item  If $\ov{x}_k =\ov{c}_k$, then $c_k^2$ gives rise to an extension  of $D_p$-modules $\bmat \phi_{\beta} & c_k^2 \\ & \Psi_2\emat$ mod $\varpi^{n_k}$, which splits, i.e., $[c_k^2]=0$.
\item  If $\ov{x}_k =\ov{d}_k$, then $d_k^2$ gives rise to an extension  of $D_p$-modules $\bmat \phi_{\beta} & d_k^2 \\ & \Psi_1\emat$ mod $\varpi^{n_k}$, which splits, i.e., $[d_k^2]=0$. 
\item  If $\ov{x}_k =\ov{f}_k$, then $f_k^1$ gives rise to an extension  of $D_p$-modules $\bmat \Psi_2 & f_k^1 \\ & \phi^{-1}_\beta\epsilon\emat$ mod $\varpi^{n_k}$, which splits, i.e., $[f_k^1]=0$.  
\end{itemize}
\end{lemma}
\begin{proof} 
 To fix attention assume that $x_k=a_k$, 
 i.e., that $\sigma_k = \bmat \Psi_2 & y_k & z_k \\ & \Psi_1 & a_k \\ && \tilde{\rho}\emat$ mod $\varpi^{n_k}$ as in Lemma \ref{LemmaUrban}.  First note that (after possibly changing to an appropriate basis for the $\tilde{\rho}$-piece and using Lemma \ref{indep1})  Siegel-ordinarity implies that \be\label{SO2} \sigma_k|_{D_p} = \bmat \Psi_2 & y_k & z_k^1&z_k^2 \\ & \Psi_1 & a_k^1 & a_k^2 \\ && \phi_{\beta}^{-1}\epsilon & h \\ 
&&& \phi_{\beta} \emat \quad \pmod{\varpi^{n_k}}.\ee

Hence we see that there indeed is a   rank 2 free $\Oo_k/\varpi^{n_k}[D_p]$-subquotient $S=\bmat \Psi_1 & a_k^1 \\ & \phi_{\beta}^{-1}\epsilon\emat$ as claimed in the Lemma. It remains to show that $S$ splits.  Assume it does not. Let $V$ be the representation space for $\sigma_k$. By Siegel-ordinarity it has a $D_p$-stable line $L$ on which $D_p$ acts via $\phi_{\beta}^{-1}\epsilon$. 
Let $\Lambda$ be a $G_{\Sigma}$-stable  lattice giving $\sigma_k$ such that $\sigma_k|_{D_p}$ mod $\varpi^{n_k}$ has the form \eqref{SO2}. Then we see by Lemma \ref{lattice1} that this $\Lambda$ must have a $D_p$-stable rank one submodule with $D_p$ action by $\phi_{\beta}^{-1}\epsilon$, hence finally $\Lambda_k:=\Lambda$ mod $\varpi^{n_k}$ must have a free $\Oo_k/\varpi^{n_k}$-submodule $\Lambda_0$ of rank one on which $D_p$ acts  by $\phi_{\beta}^{-1}\epsilon$.

We now claim that the subquotient $S$ also has a free $\Oo_k/\varpi^{n_k}$-submodule which is stabilized by $D_p$ and on which $D_p$ acts via $\phi_{\beta}^{-1}\epsilon$. 
Indeed, write $\mB=\{e_1, \dots, e_4\}$ for an $\Oo_k/\varpi^{n_k}$-basis of $\Lambda_k$ such that with respect to that basis we have $\sigma_k|_{D_p}$ in form \eqref{SO2}. Write $\Lambda'=(\Oo_k/\varpi^{n_k})e_1 \oplus (\Oo_k/\varpi^{n_k}) e_2 \oplus (\Oo_k/\varpi^{n_k}) e_3$ and $\Lambda'':= (\Oo_k/\varpi^{n_k})e_4$. We note that $\Lambda'$ is stable under the action of $D_p$. We first want to show that $\Lambda_0 \subset \Lambda'$. Let $v_0 \in \Lambda_0$ be an $\Oo_k/\varpi^{n_k}$-module generator. Using the fact that $\mB$ is a basis we can decompose $v_0$ uniquely as $v_0 = v_0' + v_0''$ with $v_0' \in \Lambda'$ and $v_0'' \in \Lambda''$. We want to show that $v_0''=0$. Let $g \in I_p$ be such that $\chi(g) \neq 1$. Then $g \cdot v_0 = \phi_{\beta}^{-1}\epsilon (g) v_0 = \epsilon(g) v_0$. On the other hand $g\cdot v_0 = g \cdot v'_0 + g\cdot v_0''$. We have that $g \cdot v'_0 \in \Lambda'$ and $g \cdot v_0'' = \phi_{\beta}(g) v_0'' + v' = v_0'' + v'$ for some $v'\in \Lambda'$. So we have $$\epsilon(g) v'_0 + \epsilon(g)v_0''= \epsilon(g) v_0 = g \cdot v_0 =  g \cdot v'_0 + v_0'' + v' \implies \epsilon(g)v_0'' - v_0'' \in \Lambda' \cap \Lambda'' = 0.$$ Since $\chi(g) \neq 1$, we see that $\epsilon(g)-1 \in (\Oo_k/\varpi^{n_k})^{\times}$, which implies that $v_0''=0$. So  $\Lambda_0 \subset \Lambda'$.

Now set $\Lambda''=(\Oo_k/\varpi^{n_k})e_1$. This is a $D_p$-stable submodule of $\Lambda'$ on which $D_p$ acts via $\Psi_2$. Notice that we have $S= \Lambda'/\Lambda''$ as $D_p$-modules. Clearly the image of $\Lambda_0\subset \Lambda'$ in $S$ is the desired  $D_p$-stable $\Oo_k/\varpi^{n_k}$-submodule of $S$ on which $D_p$ acts via $\phi_{\beta}^{-1}\epsilon$. We just need to show that this image is non-zero. Suppose to the contrary that it is zero, i.e., that $\Lambda_0 \subset \Lambda''$.  Let $d \in D_p$ be such that $\Psi_1(d) \not \equiv \phi_{\beta}^{-1}\epsilon(d)$ mod $\varpi$. Then we get $\phi_{\beta}^{-1}\epsilon(d) v_0 = d \cdot v_0 = \Psi_1(d) v_0$, which implies $v_0=0$, a contradiction. 
This now proves the claim about $S$. 

 In other words 
there must exist a matrix $A=\bmat a&b \\ c& d \emat \in \GL_2(\Oo_k)$ such that $$\bmat \Psi_1 & a_k^1 \\ & \phi_{\beta}^{-1} \epsilon\emat A = A \bmat \phi_{\beta}^{-1} \epsilon & * \\ & \Psi_1\emat \pmod{\varpi^{n_k}}.$$ Suppose that $[a_k^1]\neq 0$, i.e., that there exists $g \in D_p$ such that $\Psi_1(g)=\phi_{\beta}^{-1}\epsilon(g)=1$ but $a_k^1(g)\neq 0$. Then comparing the upper left entries of both sides evaluated at $g$ we get $a+a_k^1(g) c = a$, from which we get that $c\equiv 0$ mod $\varpi$. For the same entry, but for a general element $g'\in D_p$ such that $\phi_{\beta}^{-1} \epsilon(g') \not\equiv \Psi_1(g')$ (mod $\varpi$), we get $\Psi_1(g') a + c a_k^1(g') = a \phi_{\beta}^{-1}\epsilon(g')$. Reducing this equation mod $\varpi$ we thus conclude that $a \equiv 0$ (mod $\varpi$). This is a contradiction since $A$ is invertible. 

The other cases, i.e., where $\ov{x}_k = \ov{c}_k, \ov{d}_k, \ov{f}_k$ are handled similarly using the fact that $\Psi_1|_{D_p}$, $\Psi_2|_{D_p}$, $\phi_{\beta}^{-1}\epsilon$, $\phi_{\beta}$ are all pairwise distinct mod $\varpi$. This finishes the proof of Lemma \ref{vanishing}.
\end{proof}

We are now ready to complete the proof of Proposition  \ref{InSel}. 
Suppose first that $x_k = d_k$. Then $\sigma_k$ mod $\varpi^{n_k}$ has a subquotient $\tau = \bmat \tilde{\rho} & * \\ & \Psi_1\emat$, i.e., $\sigma_k$ mod $\varpi^{n_k}$ has a submodule $\tau = \bmat \Psi_1 \\ d_k & \tilde{\rho} \emat$ which is non-split mod $\varpi$ as $[\ov{x}_k]\neq 0$.  Thus $d_k$ gives rise to a class 
 in $H^1(G_{\Sigma}, \Hom(\Psi_1, \tilde{\rho}) \otimes \Oo_k/\varpi^n)$ such that $\varpi^{n_k-1}[d_k] \neq 0$.

  Furthermore, by Lemma \ref{vanishing} we must have  $\tau|_{D_p} \cong \bmat \Psi_1 &0&0\\ d_k^1 & \phi_{\beta}^{-1}\epsilon & h \\ 0 & 0 & \phi_{\beta} \emat.$ (Note that while $d_k^2$ as in Lemma \ref{vanishing} gives rise to an  (in fact split) extension of $\Psi_1$ by $\phi_{\beta}$, it is not necessarily true that $d_k^1$ gives rise to an  extension of $\Psi_1$ by $\phi_{\beta}^{-1}\epsilon$ because $h$ may be non-zero. However, it is still true that $d$ gives rise to an extension of $\Psi_1$ by $\tilde{\rho}$.) Recall that $\tilde{\rho} = \rho_g$.

Furthermore, in the basis giving rise to $\tau$ as above, the module $T_{g,k,n_k}$ corresponds to vectors $\bmat 0\\\alpha \\ \beta \emat$ while the  submodule $T^+_{g,k,n_k}$ of $T_{g,k,n_k}$ corresponds to vectors of the form $\bmat 0\\\alpha \\ 0 \emat \in T_{g,k,n_k}$, as on these vectors $I_p$ acts via $\epsilon$. Since in the same basis for every $\gamma \in I_p$ we have $d_k(\gamma) = \bmat d_k^1(\gamma) \\ 0 \emat$, we see that $d_k(\gamma)(\alpha) \in T_{g,k,n_k}^+$ for every $\alpha$ in the representation space of $\Psi_1$. 

 Let $A$ be an $\Oo_k$-algebra (we will only use $A=\Oo_k$ or $A=\Oo_k/\varpi^s \Oo_k$). Given two free $A$-modules  $V_1, V_2$ we have a canonical isomorphism $V_2 \otimes_A V_1^{\vee} \xrightarrow{\psi} \Hom_A(V_1, V_2)$ given by $v \otimes \phi \mapsto (v' \mapsto \phi(v') v)$. If $W_2 \subset V_2$ is an $A$-direct summand (i.e., $V_2 = W_2 \oplus W_2'$ for some $A$-submodule $W_2'$ of $V_2$) then the isomorphism $\psi$ carries $W_2 \otimes V_1^{\vee}\subset V_2  \otimes V_1^{\vee}$ onto the (direct summand of $\Hom_A(V_1, V_2)$ consisting of) homomorphisms whose image is contained in $W_2$. Similarly, if $W_1 \subset V_1$ is an $A$-direct summand and we denote by $W_1' \subset V_1^{\vee}$ the submodule (which is a direct summand) of linear functionals which kill $W_1$, then $\psi$ carries $W'_1$ onto the direct summand of $\Hom(V_1, V_2)$ consisting of homomorphisms that kill $W_1$. All of these follow immediately from the fact that tensor product as well as the $\Hom$-functor commute with direct sums in both coordinates.

Even though by Lemma \ref{Psis}   we have $\Psi_1=1$, in the argument below we keep the notation $\Psi_1$ to convince the reader that analogous calculations hold for any character, hence in particular can be applied in the cases when $\ov{x}_k=\ov{c}_k$ or $\ov{f}_k$, where $\Psi_1$ is replaced by $\Psi_2$ which is a non-trivial character.  By the above we see that for every $\gamma \in I_p$ we get that $d_k(\gamma) \in \Hom(\Psi_1, T_{g,k,n_k})$ has image contained in $T_{g,k,n_k}^+$, so the image of $d_k(\gamma)$ under the inverse of the isomorphism $\psi: T_{g,k,n_k} \otimes \Psi_1^{-1} \to \Hom_{\Oo_k/\varpi^n\Oo_k}(\Psi_1, T_{g,k,n_k})$ is an element of $T_{g,k,n_k}^+\otimes \Psi_1^{-1} \subset T_{g,k,n_k}\otimes \Psi_1^{-1}$. Thus, $d_k$ gives rise to an element of $\Sel_{1, k, n_k}$ as desired.

Suppose now that $x_k=a_k$.  Then $\sigma_k$ mod $\varpi^{n_k}$ has a submodule $\tau = \bmat \Psi_1 & a_k \\ & \tilde{\rho}\emat$ which is non-split mod $\varpi$ as $[\ov{x}_k]\neq 0$.  Thus $a_k$ gives rise to a class in $$H^1(G_{\Sigma}, \Hom(T_{g,k,n_k}, \Oo_k/\varpi^{n_k}\Oo_k(\Psi_1)))$$ such that $\varpi^{n_k-1}[a_k]\neq 0$.  Again by Lemma \ref{vanishing} we must have $\tau|_{D_p} = \bmat \Psi_1 &0&a_k^2\\ 0& \phi_{\beta}^{-1}\epsilon & h \\ 0 & 0 & \phi_{\beta} \emat.$ We will now show  that for every $\gamma \in I_p$ the homomorphism $a_k(\gamma)$ kills $T_{g,k,n_k}^+$. Indeed, note that in the basis which gives the above form of $\tau$ we have $a_k = \bmat 0 & a_k^2\emat$, while $T_{g,k,n_k}^+$ is given again by the vectors of the form $\bmat 0\\ \alpha \\ 0 \emat \in T_{g,k,n_k}$.  

By the discussion above we conclude that  the inverse of the  isomorphism $\psi: \Oo_k/\varpi^{n_k}(\Psi_1) \otimes T_{g,k,n_k}^{\vee} \to \Hom(T_{g,k,n_k}, \Oo_k/\varpi^{n_k}(\Psi_1))$ carries $a_k(\gamma)$ into the subspace $\Oo_k/\varpi^{n_k}(\Psi_1) \otimes (T_{g,k,n_k}^+)'\subset \Oo_k/\varpi^{n_k}(\Psi_1)\otimes T_{g,k,n_k}^{\vee}$, where as above $(T_{g,k,n_k}^+)'$ denotes the submodule of $T_{g,k,n_k}^{\vee}$ consisting of functionals which kill $T_{g,k,n_k}^+$.

Note that since $\Psi_1\Psi_2=\epsilon$, we get  $\Psi_1\otimes \rho_g^{\vee} \cong \Psi_2^{-1}\epsilon\otimes \rho_g^{\vee} \cong  \Psi_2^{-1} \otimes \rho_g^{\vee}(1)$. Under these isomorphisms the  module  $\Oo_k/\varpi^{n_k}(\Psi_1) \otimes (T_{g,k,n_k}^+)'$ gets mapped to $\Oo_k/\varpi^{n_k}(\Psi_2^{-1} \epsilon)\otimes  (T_{g,k,n_k}^+)'$ and finally to $\Oo_k/\varpi^{n_k}(\Psi_2^{-1}) \otimes  (T_{g,k,n_k}^+)'(1)$. Finally (by essential self-duality of $\rho_g$) there is an isomorphism of $G_{\Sigma}$-modules $\psi':\rho_g\to \rho_g^{\vee}(1)$.  
We note that  $T_{g,k,n_k}^+$ is the unique direct summand of $T_{g,k,n_k}$  which is stable under $I_p$ and such that  $I_p$ acts on it by $\epsilon$. Hence $\psi'$ (as it is $G_{\Sigma}$-equivariant) must carry $T_{g,k,n_k}^+$ onto the unique direct summand of $T_{g,k,n_k}^{\vee}(1)$ with the same property, i.e., $\psi'(T_{g,k,n_k}^+) = X \otimes \epsilon$ where $X$ is the unique direct summand of $T_{g,k,n_k}^{\vee}$ on which $I_p$ acts trivially. 

Let $\phi \in (T_{g,k,n_k}^+)'$. Let $\gamma \in I_p$, $v=\bmat v_1 \\ v_2 \emat  \in T_{g,k,n_k}$. (We suppress the $0$ from $\bmat 0\\ v_1 \\ v_2 \emat$.)  Then $$(\gamma \cdot \phi )(v) = \phi (\rho_g(\gamma^{-1})v)= \phi(\bmat \epsilon(\gamma)^{-1} & h(\gamma^{-1})\\ & 1\emat v) =\phi\left(\bmat \epsilon(\gamma)^{-1} v_1 + h(\gamma^{-1}) v_2 \\ v_2 \emat\right)= v_2 = \phi(v).$$ Hence $I_p$ acts trivially on $(T_{g,k,n_k}^+)'$, i.e., we must have $X=(T_{g,k,n_k}^+)'$. In other words $\psi'$ carries $T_{g,k,n_k}^+$ onto $(T_{g,k,n_k}^+)'(1)$. This proves that for $\gamma \in I_p$ we have that $a_k(\gamma)$ is mapped under $\psi^{-1}$ into $\Oo_k/\varpi^{n_k}(\Psi_1)\otimes (T_{g,k,n_k}^+)' \cong \Oo_k/\varpi^{n_k}(\Psi_2^{-1}) \otimes (T_{g,k,n_k}^+)'(1)$ and further mapped under $(\psi')^{-1}$ into the the direct summand $\Oo_k/\varpi^{n_k}(\Psi_2^{-1})\otimes T_{g,k,n_k}^+\subset \Oo_k/\varpi^{n_k}(\Psi_2^{-1})\otimes T_{g,k,n_k}$. Hence we get $[a_k] \in \Sel_{2, k, n_k}$. 

The cases $\ov{x}_k = \ov{c}_k$ and $\ov{x}_k = \ov{f}_k$ are handled in an analogous way. 
Finally the fact that $\varpi^{n_k-1}[x_k]\neq 0$ follows from Lemma \ref{LemmaUrban}.
\end{proof}

\begin{cor} \label{old scalars}  If $\ov{x}_k\in \{\ov{d}_k, \ov{f}_k\}$, then there exists an element $x'_k \in {\rm Sel}_1$ such that $\varpi^{n_k-1}x'_k \neq 0$. If, on the other hand,  $\ov{x}_k\in \{\ov{a}_k, \ov{c}_k\}$, then there exists an element $x'_k \in {\rm Sel}_2$ such that $\varpi^{n_k-1}x'_k \neq 0$. 
\end{cor}

\begin{proof}
First note that as the formation of Selmer groups commutes with direct sums of Galois modules and $\Oo_k/\varpi^r =(\Oo/\varpi^r)^{s}$ where $s= [\Oo_k:\Oo]$ one has  ${\rm Sel}_{i, k, n_k} = \left({\rm Sel}_{i, 2, n_k}\right)^s$.  If $\ov{x}_k\in \{\ov{d}_k, \ov{f}_k\}$ then by Proposition \ref{InSel} we get that $[x_k] \in {\rm Sel}_{1, k, n_k}$ is such that $\varpi^{n_k-1} [x_k] \neq 0$. Thus there must exist an element $x'_k \in {\rm Sel}_{1, 2, n_k}$ which is not annihilated by $\varpi^{n_k-1}$. As we have an inclusion ${\rm Sel}_{1, 2, n_k} \hookrightarrow {\rm Sel}_1[\varpi^{n_k}]$, we can regard $x_k'$ as an element of $ {\rm Sel}_1$ which is not killed by $\varpi^{n_k-1}$. The other case is analogous.
\end{proof}
We are now ready to finish the proof of Theorem \ref{SK case}, i.e.,   that the pseudo-representation $T$ is not of Saito-Kurokawa type. Indeed, we will now arrive at a contradiction.  Since by Lemma \ref{inde} for every $k \in \mS$ there exists $\ov{x}_k \in \{\ov{a}_k, \ov{c}_k, \ov{d}_k, \ov{f}_k\}$ such that $[\ov{x}_k]$ gives rise to a non-split extension of the corresponding Jordan-Holder blocks of $1 \oplus \rho \oplus \chi$, there exists $A\in \{a,c,d,f\}$ and an infinite subsequence $\mT \subset \mS$ such that for all $k \in \mT$ we have that $[\ov{x}_k]=[\ov{A}_k]$ is such a non-split extension. Fix such an $A$. Then Proposition \ref{InSel}  gives us an extension $[A_k] \in {\rm Sel}_{i, k, n_k}$ for $i=1$ or 2 such that $\varpi^{n_k-1} [A_k] \neq 0$. Set $i(A)=1$ if the extension $[A_k]$ lies in $\Sel_{1, k, n_k}$ and $i(A)=2$ if the extension $[A_k]$ lies in $\Sel_{2,k, n_k}$. Then by Corollary \ref{old scalars} we get an element $A'_k\in {\rm Sel}_{i(A)}$ not annihilated by $\varpi^{n_k-1}$. As $n_k$ tends to $\infty$ for $k \in \mT$, we see that $\Sel_{i(A)}$ must be infinite. Thus we obtain a contradiction to Proposition \ref{Sel finite}.

\section{Ruling out Yoshida type} \label{Yoshida type}

In this section we show that $\sigma_2$ is not the direct sum of two irreducible two-dimensional representations under some assumptions.

For a positive integer $N$  we will write $S_2^{(2)}(\Gamma^{\rm para}(N))$ for weight 2 genus 2 Siegel modular forms of paramodular level $N$.

\begin{prop} \label{Y}
Suppose at least one of the following holds:
\begin{enumerate}[(I)]

\item One has $\ell \not\equiv \pm 1$ mod $p$ for all $\ell \mid N$ and  $\sigma_2$ is Borel-ordinary at $p$.
\item One has  $\ell \not\equiv \pm 1$ mod $p$ for all $\ell \mid N$ and $\sigma_2$ is crystalline at $p$.
\item One has $p>3$ and $\sigma_2=\sigma_F$ for some classical Siegel modular form $F \in S_2^{(2)}(\Gamma^{\rm para}(N))$  which has distinct roots for its Hecke polynomial at $p$.
\end{enumerate}
Then $\sigma_2$ is not of Yoshida type.
\end{prop}

\begin{proof}
Assume that in fact $\sigma_2 = \rho_1 \oplus \rho_2$ with $\rho_1, \rho_2$ irreducible and $\ov{\rho}_1=\rho$ and $\ov{\rho}_2^{\rm ss} = 1 \oplus \chi$.  
By Lemma \ref{HT}(i) we have $(\sigma_2|_{D_p})^{\rm ss}=\phi_{\beta}^{-1}\epsilon \oplus \phi_{\beta} \oplus \gamma$ , which  as in Lemma \ref{ordrhotilde} implies that $\rho_1$ is ordinary, i.e., that $\rho_1|_{D_p} \cong_E \bmat \phi_{\beta}^{-1} \epsilon & * \\ & \phi_{\beta}\emat$. By Lemma \ref{HT}(ii) the Hodge-Tate-Sen weights of $\sigma_2$ are 0,0,1,1. 

Proof of (I): As $\sigma_2$ is Borel-ordinary,  this forces $\rho_2|_{D_p}$ to be ordinary, i.e., $\rho_2|_{D_p} \cong \bmat  \phi_{\alpha}^{-1}\epsilon &* \\  & \phi_{\alpha}\emat$ for some $\alpha\in \Oo^{\times}$. On the other hand since $\rho_2$ is irreducible there exists a $G_{\Sigma}$-stable lattice $\Lambda$  in the space of $\rho_2$  such that with respect to that lattice we have \be \label{form2} \ov{\rho}_{2, \Lambda} =\bmat 1 & a \\ & \chi\emat \not\cong 1 \oplus \chi.\ee  By Lemma \ref{lattice1}, the lattice $\Lambda$ must have a $D_p$-stable line on which $D_p$ acts via $\phi_{\alpha}^{-1}\epsilon$, so $\ov{\rho}_{2, \Lambda}|_{D_p} \cong \bmat \ov{\phi}_{\alpha}^{-1} \chi & * \\ & \ov{\phi}_{\alpha} \emat$. By comparing with the form \eqref{form2} and using that $\chi$ is ramified we conclude that $\ov{\phi}_{\alpha}=1$, so in fact $\ov{\rho}_{2, \Lambda}|_{D_p} \cong \bmat \chi & * \\ & 1 \emat$. Thus $\ov{\rho}_2|_{D_p} \cong 1 \oplus \chi$. This in particular implies that $\ov{\rho}_2$ splits when restricted to $I_p$. Hence  $a$ gives rise to a class in $$H^1_{\Sigma}(\bfQ, \bfF(-1)):=\ker(H^1(G_{\Sigma}, \bfF(-1)) \overset{{\rm res}_{p}}{\to} H^1(I_p, \bfF(-1))).$$ 
Since $\ell \not\equiv \pm 1$ mod $p$ for all $\ell\mid N$ we use Lemma 6.3 in \cite{BergerKlosin19accepted} to conclude that  $H^1_{\Sigma}(\bfQ, \bfF(-1)) =\ker(H^1(G_{\Sigma}, \bfF(-1)) \to \prod_{\ell \in \Sigma} H^1(I_{\ell}, \bfF(-1)))$. This part of the class group of $\bfQ(\mu_p)$  is zero  by Proposition 6.16 in \cite{Washingtonbook}. This implies that $\ov{\rho}_{2, \Lambda}$ is split which leads to a contradiction. 

Proof of (II): As before there exists a $G_{\Sigma}$-stable lattice $\Lambda$ such that with respect to that lattice we have $\ov{\rho}_{2, \Lambda} =\bmat 1 & a \\ & \chi\emat \not\cong 1 \oplus \chi$. Since $\sigma_2$ is crystalline and its Hodge-Tate-Sen weights are 0,0,1,1, it is in the Fontaine-Laffaille range. Hence so is $\rho_2$. This implies (see e.g. \cite{BergerKlosin19accepted} Lemma 6.1) that the extension given by $a$ gives rise to a non-zero element in $H^1_{\Sigma}(\bfQ, \bfF(-1))$, 
which again gives a contradiction as $H^1_{\Sigma}(\bfQ, \bfF(-1))=0$.

Proof of (III): We have $\sigma_2=\sigma_F$ for some classical Siegel modular form $F \in S_2^{(2)}(\Gamma^{\rm para}(N))$. We can assume that $F$ is not a Saito-Kurokawa lift (as then $\tr \sigma_F$ would not be of Yoshida type). By \cite{SchmidtCAP} this means that $F$ is of type $(G)$. The assumption on the roots of the Hecke polynomial implies by \cite{Jorza10} Theorem 4.3.4 or \cite{Mok14} Proposition 4.16 that $\sigma_2$ is crystalline at $p$.   If $\ell \not\equiv \pm 1$ mod $p$ for all $\ell \mid N$  then we get a contradiction as in (I) and (II). Without this assumption  we argue as in the proof of \cite{BergerKlosinAJMaccepted} Theorem 8.6, i.e. apply \cite{Ramakrishnan13} Theorem C and \cite{KudlaRallisSoudry92} Theorem 7.1 to deduce that $F$ would have to be of Yoshida type, i.e. not of type (G),  a contradiction.

\begin{rem}\label{limitation}

 Note that the key issue in the Yoshida case  is ruling out that $\sigma_2$ is the sum of an (ordinary) 2-dimensional Galois representation associated to a classical form (with associated $\mod{p}$-representation $\rho$) and a 2-dimensional Galois representation that is a priori not de Rham.

 It is worth noting that whilst we are able to rule out that $\sigma_2$ is of Saito-Kurokawa type only using properties of the representations $\sigma_k$ for $k \in \mathcal{S}$ the Yoshida type case requires additional information about $\sigma_2$. 
In particular, while for both the Saito-Kurokawa  and the Yoshida type we assume crystallinity of the representations $\sigma_k$, in case (II) of Proposition \ref{Y} we also need to assume that $\sigma_2$ itself is crystalline.  On the other hand, work in progress by Ariel Weiss shows that a classical Siegel-ordinary type (G) eigenform has irreducible Galois representation. This would allow us to drop the assumption in (III) on the distinctness of the roots of the Hecke polynomial.
\end{rem}

\end{proof}

\bibliographystyle{amsalpha}
\bibliography{standard2}

\newcommand{\etalchar}[1]{$^{#1}$}
\providecommand{\bysame}{\leavevmode\hbox to3em{\hrulefill}\thinspace}
\providecommand{\MR}{\relax\ifhmode\unskip\space\fi MR }
\providecommand{\MRhref}[2]{%
  \href{http://www.ams.org/mathscinet-getitem?mr=#1}{#2}
}
\providecommand{\href}[2]{#2}
\begin{thebibliography}{{LMF}20}

\bibitem[AB14]{AgarwalBrown14}
M.~Agarwal and J.~Brown, \emph{On the {B}loch-{K}ato conjecture for elliptic
  modular forms of square-free level}, Math. Z. \textbf{276} (2014), no.~3-4,
  889--924.

\bibitem[AIP15]{AndreattaIovitaPilloni15}
F.~Andreatta, A.~Iovita, and V.~Pilloni, \emph{{$p$}-adic families of {S}iegel
  modular cuspforms}, Ann. of Math. (2) \textbf{181} (2015), no.~2, 623--697.

\bibitem[Art04]{Arthur04}
J.~Arthur, \emph{Automorphic representations of {${\rm GSp(4)}$}},
  Contributions to automorphic forms, geometry, and number theory, Johns
  Hopkins Univ. Press, Baltimore, MD, 2004, pp.~65--81.

\bibitem[BB19]{BergerBetina19}
T.~Berger and A.~Betina, \emph{On {S}iegel eigenvarieties at {S}aito-{K}urokawa
  points}, preprint (2019), available at
  \url{https://arxiv.org/abs/1902.05885}.

\bibitem[BC09]{BellaicheChenevierbook}
J.~Bella{\"{\i}}che and G.~Chenevier, \emph{$p$-adic families of {G}alois
  representations and higher rank {S}elmer groups}, Ast\'erisque (2009),
  no.~324.

\bibitem[BCGP18]{BCGP}
G.~Boxer, F.~Calegari, T.~Gee, and V.~Pilloni, \emph{Abelian surfaces over
  totally real fields are potentially modular}, preprint available at
  \url{http://www.math.uchicago.edu/~fcale/papers/surfaces.pdf}.

\bibitem[Bel03]{Bellaiche03}
J.~Bella\"{\i}che, \emph{\`a propos d'un lemme de {R}ibet}, Rend. Sem. Mat.
  Univ. Padova \textbf{109} (2003), 45--62.

\bibitem[BK13]{BergerKlosin13}
T.~Berger and K.~Klosin, \emph{On deformation rings of residually reducible
  {G}alois representations and ${R}={T}$ theorems}, Math. Ann. \textbf{355}
  (2013), no.~2, 481--518.

\bibitem[BK14]{BrumerKramer14}
A.~Brumer and K.~Kramer, \emph{Paramodular abelian varieties of odd conductor},
  Trans. Amer. Math. Soc. \textbf{366} (2014), no.~5, 2463--2516.

\bibitem[BK19]{BergerKlosin19accepted}
T.~Berger and K.~Klosin, \emph{Modularity of residual {G}alois extensions and
  the {E}isenstein ideal}, Trans. Amer. Math. Soc. \textbf{372} (2019), no.~11,
  8043--8065.

\bibitem[BKar]{BergerKlosinAJMaccepted}
\bysame, \emph{Deformations of {S}aito-{K}urokawa type and the {P}aramodular
  {C}onjecture}, Amer. J. Math. (to appear), arXiv:1710.10228.

\bibitem[BL19]{BrownLi19preprint}
J.~Brown and H.~Li, \emph{Congruence primes for automorphic forms on symplectic
  groups}, preprint (2019).

\bibitem[Bro07]{Brown07}
J.~Brown, \emph{Saito-{K}urokawa lifts and applications to the {B}loch-{K}ato
  conjecture}, Compos. Math. \textbf{143} (2007), no.~2, 290--322.

\bibitem[Bro08]{Brown08}
\bysame, \emph{Residually reducible representations of algebras over local
  {A}rtinian rings}, Proc. Amer. Math. Soc. \textbf{136} (2008), no.~10,
  3409--3414.

\bibitem[BSS{\etalchar{+}}16]{BSSVY}
A.~Booker, J.~Sijsling, A.~Sutherland, J.~Voight, and D.~Yasaki, \emph{A
  database of genus-2 curves over the rational numbers}, LMS Journal of
  Computation and Mathematics \textbf{19} (2016), no.~A, 235–254.

\bibitem[DDT97]{DDT}
H.~Darmon, F.~Diamond, and R.~Taylor, \emph{Fermat's last theorem}, Elliptic
  curves, modular forms \& Fermat's last theorem (Hong Kong, 1993), Internat.
  Press, Cambridge, MA, 1997, pp.~2--140.

\bibitem[Dia96]{Diamond96}
F.~Diamond, \emph{On deformation rings and {H}ecke rings}, Ann. of Math. (2)
  \textbf{144} (1996), no.~1, 137--166.

\bibitem[Gha05]{Ghate05}
E.~Ghate, \emph{Ordinary forms and their local {G}alois representations},
  Algebra and number theory, Hindustan Book Agency, Delhi, 2005, pp.~226--242.

\bibitem[GT19]{GeeTaibi19}
T.~Gee and O.~Ta\"{\i}bi, \emph{Arthur's multiplicity formula for {${\bf
  GSp}_4$} and restriction to {${\bf Sp}_4$}}, J. \'{E}c. polytech. Math.
  \textbf{6} (2019), 469--535.

\bibitem[Jor10]{Jorza10}
A.~Jorza, \emph{{C}rystalline representations for {${\rm GL}_2$} over quadratic
  imaginary fields}, Ph.D. thesis, Princeton University, 2010.

\bibitem[Kaw10]{Kawamura10preprint}
H.-A. Kawamura, \emph{On certain constructions of $p$-adic {S}iegel modular
  forms of even genus}, preprint available at
  \url{https://arxiv.org/pdf/1011.6476.pdf} (2010).

\bibitem[KRS92]{KudlaRallisSoudry92}
S.~S. Kudla, S.~Rallis, and D.~Soudry, \emph{On the degree {$5$} {$L$}-function
  for {${\rm Sp}(2)$}}, Invent. Math. \textbf{107} (1992), no.~3, 483--541.

\bibitem[{LMF}20]{lmfdb}
The {LMFDB Collaboration}, \emph{The {L}-functions and modular forms database},
  \url{http://www.lmfdb.org}, 2020, [Online; accessed 13 January 2020].

\bibitem[Mak18]{Makiyama18}
K.~Makiyama, \emph{{O}n $p$-adic families of the ${D}$-th {S}aito-{K}urokawa
  lifts}, preprint available at
  \url{http://www.kurims.kyoto-u.ac.jp/~kyodo/kokyuroku/contents/pdf/2100-01.pdf}.

\bibitem[Miy89]{Miyake89}
T.~Miyake, \emph{Modular forms}, Springer-Verlag, Berlin, 1989, Translated from
  the Japanese by Yoshitaka Maeda.

\bibitem[Mok14]{Mok14}
C.~P. Mok, \emph{Galois representations attached to automorphic forms on {${\rm
  GL}_2$} over {CM} fields}, Compos. Math. \textbf{150} (2014), no.~4,
  523--567.

\bibitem[Ram13]{Ramakrishnan13}
D.~Ramakrishnan, \emph{Decomposition and parity of {G}alois representations
  attached to {$\rm GL(4)$}}, Automorphic representations and {$L$}-functions,
  Tata Inst. Fundam. Res. Stud. Math., vol.~22, Tata Inst. Fund. Res., Mumbai,
  2013, pp.~427--454.

\bibitem[Sch07]{Schmidt07}
R.~Schmidt, \emph{On classical {S}aito-{K}urokawa liftings}, J. Reine Angew.
  Math. \textbf{604} (2007), 211--236.

\bibitem[Sch18]{SchmidtCAP}
\bysame, \emph{Paramodular forms in {CAP} representations of {${\rm GSp}_4$}},
  preprint available at \url{http://www.math.unt.edu/~schmidt/}.

\bibitem[Sor09]{Sorensen09}
C.~M. Sorensen, \emph{Level-raising for {S}aito-{K}urokawa forms}, Compos.
  Math. \textbf{145} (2009), no.~4, 915--953.

\bibitem[SU06]{SkinnerUrban06}
C.~Skinner and E.~Urban, \emph{Sur les d\'eformations {$p$}-adiques de
  certaines repr\'esentations automorphes}, J. Inst. Math. Jussieu \textbf{5}
  (2006), no.~4, 629--698.

\bibitem[SU14]{SkinnerUrban14}
\bysame, \emph{The {I}wasawa {M}ain {C}onjectures for {$GL_2$}}, Invent. Math.
  \textbf{195} (2014), no.~1, 1--277.

\bibitem[Tay91]{Taylor91}
R.~Taylor, \emph{Galois representations associated to {S}iegel modular forms of
  low weight}, Duke Math. J. \textbf{63} (1991), no.~2, 281--332.

\bibitem[Til06]{Tilouine06}
J.~Tilouine, \emph{Nearly ordinary rank four {G}alois representations and
  {$p$}-adic {S}iegel modular forms}, Compos. Math. \textbf{142} (2006), no.~5,
  1122--1156, With an appendix by Don Blasius.

\bibitem[Urb99]{Urban99}
E.~Urban, \emph{On residually reducible representations on local rings}, J.
  Algebra \textbf{212} (1999), no.~2, 738--742.

\bibitem[Was97]{Washingtonbook}
L.~C. Washington, \emph{Introduction to cyclotomic fields}, second ed.,
  Graduate Texts in Mathematics, vol.~83, Springer-Verlag, New York, 1997.

\bibitem[Wei19]{Weiss19}
A.~Weiss, \emph{On {G}alois representations associated to low weight
  {H}ilbert-{S}iegel modular forms}, Ph.D. thesis, University of Sheffield,
  2019, available at \url{http://etheses.whiterose.ac.uk/25130/}.

\end{thebibliography}

\end{document}